\documentclass[reqno,11pt]{amsart}

\usepackage{amsthm}
\usepackage{amssymb}
\usepackage{amsmath}
\usepackage{amsfonts}

\usepackage{breqn}
\usepackage{graphicx,psfrag}

\usepackage{subcaption}
\usepackage{comment}

\usepackage[utf8]{inputenc}
\usepackage[english]{babel}

\usepackage[T1]{fontenc}
\usepackage{lmodern}

\usepackage{bookmark}

\usepackage{enumitem}
\def\rmlabel{\upshape({\itshape \roman*\,})}

\usepackage{xcolor}
\usepackage{hyperref}

\hypersetup{%
    pdfauthor = {Maycon Sambinelli},%
    pdftitle = {My title},%
    pdfkeywords = {key words},%
    pdfproducer = {Latex with hyperref},%
    backref = page,%
    pdfcreator = {pdflatex},%
    colorlinks,%
    linkcolor = {red!60!black},%
    citecolor = {green!60!black},%
    urlcolor = {blue!60!black}%
}

\usepackage{tikz}
\usetikzlibrary{decorations.markings}
\usetikzlibrary{calc,positioning,decorations.pathmorphing,decorations.pathreplacing,shapes.geometric}

\makeatletter
\let\origsection=\section %
\def\section{\@ifstar{\origsection*}{\mysection}}
\def\mysection{\@startsection{section}{1}\z@{.7\linespacing\@plus\linespacing}{.5\linespacing}{\normalfont\scshape\centering\S}}
\makeatother

\usepackage[abbrev,msc-links,backrefs]{amsrefs}
\usepackage{doi}

\renewcommand{\PrintDOI}[1]{\doi{#1}}

\usepackage{datetime}
\usepackage{lineno}
\newcommand*\patchAmsMathEnvironmentForLineno[1]{%
\expandafter\let\csname old#1\expandafter\endcsname\csname #1\endcsname
\expandafter\let\csname oldend#1\expandafter\endcsname\csname end#1\endcsname
\renewenvironment{#1}%
{\linenomath\csname old#1\endcsname}%
{\csname oldend#1\endcsname\endlinenomath}}%
\newcommand*\patchBothAmsMathEnvironmentsForLineno[1]{%
\patchAmsMathEnvironmentForLineno{#1}%
\patchAmsMathEnvironmentForLineno{#1*}}%
\AtBeginDocument{%
\patchBothAmsMathEnvironmentsForLineno{equation}%
\patchBothAmsMathEnvironmentsForLineno{align}%
\patchBothAmsMathEnvironmentsForLineno{flalign}%
\patchBothAmsMathEnvironmentsForLineno{alignat}%
\patchBothAmsMathEnvironmentsForLineno{gather}%
\patchBothAmsMathEnvironmentsForLineno{multline}%
}

\usepackage{geometry}
\newcommand{\oldqed}{}
\def\endofFact{\hfill\scalebox{.6}{$\Box$}}
\newenvironment{claimproof}[1][Proof]{
  \renewcommand{\oldqed}{\qedsymbol}
  \renewcommand{\qedsymbol}{\endofFact}
  \begin{proof}[#1]
}{
  \end{proof}
  \renewcommand{\qedsymbol}{\oldqed}
} 

\newtheorem{theorem}{Theorem}[section]
\newtheorem{lemma}[theorem]{Lemma}
\newtheorem{proposition}[theorem]{Proposition}
\newtheorem{conjecture}[theorem]{Conjecture}

\newtheorem{claim}[theorem]{Claim}

\newtheorem{fact}[theorem]{Fact}

\theoremstyle{definition}
\newtheorem{definition}[theorem]{Definition}

\theoremstyle{case}

\numberwithin{equation}{section}

\usepackage{xspace}

\newcommand{\floor}[1]{\lfloor #1 \rfloor}
\newcommand{\ceil}[1]{\lceil #1 \rceil}

\newcommand\tand{\ \text{and}\ }
\let\:\colon

\newcommand{\sD}{\mathcal{D}}

\newcommand{\ini}{{\rm ini}}

\tikzset{red vertex/.style={circle,draw,minimum size=2mm,inner sep=0pt,outer sep=4pt,fill=red, color=red}}
\tikzset{blue vertex/.style={circle,draw,minimum size=2mm,inner sep=0pt,outer sep=4pt,fill=blue, color=blue}}
\tikzset{black vertex/.style={circle,draw,minimum size=2mm,inner sep=0pt,outer sep=3pt,fill=black, color=black}}
\tikzset{white vertex/.style={circle,draw,minimum size=2mm,inner sep=0pt,outer sep=3pt, color=black}}
\tikzset{green vertex/.style={circle,draw,minimum size=2mm,inner sep=0pt,outer sep=3pt, color=green, fill=green}}

\tikzstyle{color1}=[color=black] 
\tikzstyle{color2}=[color=blue]
\tikzstyle{color3}=[color=red]
\tikzstyle{color4}=[color=green]

\tikzstyle{edge}=[line width=0.9]

\tikzstyle{possible edge}=[edge, dashed]

\tikzstyle{fucking loosely dotted}=[dash pattern=on \pgflinewidth off 6pt]
\tikzstyle{nonedge}=[color=red, fucking loosely dotted, line width=1.3]

\tikzstyle{snake}=[decorate, decoration=snake, segment length=1cm]
\tikzstyle{short snake}=[decorate, decoration=snake, segment length=7mm]

\colorlet{setfilling}{green!5!white}
\colorlet{setborder}{gray}

\newcommand{\chii}{\chi'_{\rm irr}}
\def\rmlabel{\upshape({\itshape \roman*\,})}
\def\red{\text{\rm red}}
\def\blue{\text{\rm blue}}
\def\green{\text{\rm green}}
\def\qand{\quad\text{and}\quad}

\let\phi=\varphi

\begin{document}

\title[Decomposing split graphs into locally irregular graphs]{Decomposing split graphs into\\ locally irregular graphs}

\author{C. N. Lintzmayer, G. O. Mota, and M. Sambinelli}

\shortdate
\yyyymmdddate
\settimeformat{ampmtime}
\date{\today, \currenttime}

\address{Center for Mathematics. Computing and Cognition. Federal University of ABC. Santo Andr{\'e}, S{\~a}o Paulo, Brazil}
\email{\href{mailto:carla.negri@ufabc.edu.br}{\texttt{carla.negri@ufabc.edu.br}}, \href{mailto:g.mota@ufabc.edu.br}{\texttt{g.mota@ufabc.edu.br}}}

\address{Institute of Mathematics and Statistics. University of S{\~a}o Paulo. S{\~a}o Paulo, Brazil}
\email{\href{mailto:sambinelli@ime.usp.br}{\texttt{sambinelli@ime.usp.br}}}

\thanks{G. O. Mota was supported by CNPq (304733/2017-2, 428385/2018-4) and FAPESP (2018/04876-1).
  M. Sambinelli was supported by FAPESP (2017/23623-4).
  This study was financed in part by the Coordenação de Aperfeiçoamento de Pessoal de Nível Superior, Brasil (CAPES), Finance Code 001.
  The research that led to this paper started in  the Workshop Paulista em Otimização, Combinatória e Algoritmos (WoPOCA) 2018, which was financed by FAPESP (2013/03447-6) and CNPq (456792/2014-7).
  FAPESP is the S\~ao Paulo Research Foundation, CAPES is the Coordination for the Improvement of Higher Education Personnel, and CNPq is the National
  Council for Scientific and Technological Development of Brazil.
}

\begin{abstract}
    A graph is locally irregular if any pair of adjacent vertices have distinct degrees.
    A locally irregular decomposition of a graph $G$ is a decomposition $\sD$ of $G$ such that every subgraph $H \in \sD$ is locally irregular.
    A graph is said to be decomposable if it admits a locally irregular decomposition.
    We prove that any decomposable split graph can be decomposed into at most three locally irregular subgraphs and we characterize all split graphs whose decomposition can be into one, two or three locally irregular subgraphs.
\end{abstract}

\maketitle

\section{Introduction}
\label{sec:introduction}

We assume that all the graphs in this text are finite and simple.
Terminology and notation used here are standard, and for missing definition we refer the reader to~\cite{Bo98,BondyMurty2008,Di10}.
Given a graph $G$, a collection $\sD = \{H_1, \ldots, H_k\}$ of subgraphs of $G$ is a \emph{decomposition} of $G$ if $\{E(H_1), \ldots, E(H_k)\}$ is a partition of $E(G)$.
A graph is \emph{locally irregular} if any pair of adjacent vertices have distinct degrees.
A \emph{locally irregular decomposition} of a graph $G$ is a decomposition $\sD$ of $G$ such that every subgraph $H \in \sD$ is locally irregular.
Not all graphs admit a locally irregular decomposition; take for example the complete graph with three vertices.
We say that a graph is \emph{decomposable} if it admits a locally irregular decomposition.
Given a decomposable graph $G$, the \emph{irregular chromatic index} of $G$, denoted by $\chii(G)$, is the smallest size of a locally irregular decomposition of $G$.
Alternatively, one can see a locally irregular decomposition of a graph $G$ as an edge coloring of $G$ such that each color induces a locally irregular graph.
We call such coloring of a \emph{locally irregular edge coloring}.

The problem of determining the irregular chromatic index of graphs is closely related to the 1-2-3 Conjecture posed by Kar\'onski, {\L}uczak and Thomason~\cite{KaLuTh04}, which states that a graph $G=(V,E)$ can be made locally irregular by changing some edges in $E$ by two or three parallel edges.

Locally irregular decomposition as defined above was introduced by Baudon, Bensmail, Przyby\l{}o, and Wo\'{z}niak~\cite{BaBePrWo15}.
They characterized all graphs that are decomposable and proved that $d$-regular graphs with $d \geq 10^7$ admits a locally irregular 3-edge coloring.
In~\cite{BaBePrWo15} they also describe a locally irregular 3-edge coloring for trees that are not an odd-length path, and for $K_n$ with $n \geq 4$, and showed a locally irregular 2-edge coloring for regular bipartite graphs with minimum degree at least 3.
Furthermore, they proved that $\chii(G) \leq \floor{|E(G)|/2}$ for all decomposable graphs $G$ and posed the following conjecture.

\begin{conjecture}[Baudon, Bensmail, Przyby\l{}o, and Wo\'{z}niak, 2015~\cite{BaBePrWo15}]
\label{conj:chi_leq_3}
    If $G$ is a decomposable graph, then $\chii(G) \leq 3$.
\end{conjecture}

Although one can check in polynomial time whether a graph $G$ is locally irregular, deciding if there exists a locally irregular 2-edge coloring of $G$ is NP-complete, even when restricted to planar graphs with maximum degree at most 6~\cite{BaBeSo15}.
Note that proving Conjecture~\ref{conj:chi_leq_3} would show that deciding whether there exists a locally irregular $k$-edge coloring for $k \geq 3$ is in P.

A result from Przyby\l{}o~\cite{Pr17} shows that every graph with minimum degree at least $10^{10}$ admits a locally irregular 3-edge coloring.
Bensmail, Merker, and Thomassen~\cite{BeMeTh17} gave the first constant upper bound on $\chii(G)$ for general decomposable graphs $G$, showing that $\chii(G) \leq 328$.
They also showed that $\chii(G) \leq 2$ for every 16-edge-connected bipartite graph $G$, and for bipartite graphs $G$ they obtained the bound $\chii(G) \leq 10$.
Lu\v{z}ar, Przyby\l{}o, and Sot\'ak~\cite{LuPrSo18} improved these results by showing that $\chi(G) \leq 220$ for any decomposable graph $G$, and that $\chii(G) \leq 7$ for any bipartite graph $G$.
They also showed that if $G$ is subcubic, then $\chi(G) \leq 4$.

\subsection{Split graphs}

A graph $G$ is \emph{split} if there exists a partition $\{X, Y\}$ of $V(G)$ such that $G[X]$ is a complete graph and $Y$ is a stable set.
We show that every decomposable split graph $G$ has $\chii(G) \leq 3$ and we characterize all split graphs $G$ with $\chii(G) = 1$, $\chii(G) = 2$ and $\chii(G) = 3$.

When defining a split graph $G$, it may be useful to write $G(X,Y)$ to also define a partition $\{X, Y\}$ of~$V(G)$, where $X$ is a maximal clique and $Y = V(G) \setminus X$ is a stable set.

Let $G(X,Y)$ be a split graph with $X = \{v_1, \ldots, v_n\}$.
For any~$v_i \in X$, we denote by $d_G(v_i, Y)$ the number of neighbors of $v_i$ in $Y$, i.e., $d_G(v_i, Y) = |N_G(v_i) \cap Y|$.
For simplicity we just write~$d_i$ for $d_G(v_i, Y)$ whenever the graph $G$ and the stable set $Y$ are clear from the context.

It is easy to verify that for split graphs $G(X,Y)$ with $X = \{v_1, \ldots, v_n\}$, we have $\chii(G) = 1$ if and only if $d_1 > \cdots > d_n$.
In fact, since $X$ is a maximal clique, $d_G(y) \leq n-1$ for all $y \in Y$, which implies that $d_G(y)$ is smaller than the degree of all its neighbors in $X$.
Therefore, $G$ is locally irregular if and only if the vertices in $X$ have distinct degrees in $G$, which is possible if and only if $d_1 > \cdots > d_n$, and hence the result follows.
We state this in the following fact.

\begin{fact}
\label{fact:splitlarge-0}
    Let $G(X,Y)$ be a split graph with $X = \{v_1, \ldots, v_n\}$ where $d_1 \geq \cdots \geq d_n$ and $n \geq 2$.
    We have $\chii(G) = 1$ if and only if $d_1 > \cdots > d_n$.
\end{fact}

Our main result is Theorem~\ref{thm:main}, which describes $\chii(G)$ for all split graphs $G(X,Y)$ such that $X$ is a clique with at least~$10$ vertices.

\begin{theorem}
\label{thm:main}
    Let $G(X,Y)$ be a split graph with $X = \{v_1, \ldots, v_n\}$ where $d_1 \geq \cdots \geq d_n$.
    If $n \geq 10$, then the following holds.
    \begin{enumerate}[label=\rmlabel]
        \item $\chii(G) \leq 2$ if and only if $d_1 \geq \floor{n/2}$ or $d_2 \geq 1$;
        \item $\chii(G) = 3$ if and only if $d_1 < \floor{n/2}$ and $d_2 = 0$.
    \end{enumerate}
\end{theorem}

Theorem~\ref{thm:main} is proved in Section~\ref{sec:large-split}, and the result about split graphs with a very small maximal clique is discussed in Section~\ref{sec:small-split} (see Theorem~\ref{thm:small}).

In the remainder of the paper, given a graph $G$ and a coloring $\phi \colon E(G) \to \{\red,\blue\}$, we denote the two edge-disjoint spanning monochromatic subgraphs under $\phi$ by $G_{\red,\phi}$ and $G_{\blue,\phi}$.
Formally,
\begin{equation*}
    G_{\red,\phi} = \big(V(G), \phi^{-1}(\red)\big)
    \qand
    G_{\blue,\phi} = \big(V(G), \phi^{-1}(\blue)\big) \enspace .
\end{equation*}
We may omit the term $\phi$ from $G_{\red,\phi}$ and $G_{\blue,\phi}$ whenever $\phi$ is clear from the context.
This notation naturally extends to colorings that use more than two colors.

\section{Decomposing split graphs with a large maximal clique}%
\label{sec:large-split}

In this section we give a characterization of the irregular chromatic index of all split graphs with a maximal clique that has at least~$10$ vertices.

Let $G(X,Y)$ with $X = \{v_1, \ldots, v_n\}$ with $n \geq 10$.
In Lemma~\ref{lemma:splitlarge-1} we prove that $d_1 < \floor{n/2}$ and $d_2 = 0$ implies $\chii(G) = 3$.
We also prove that if $d_1 \geq \floor{n/2}$ or $d_2 \geq 1$, then $\chii(G) \leq 2$, which follows directly from Lemmas~\ref{lemma:splitlarge-2} and~\ref{lemma:splitlarge-3}.
Therefore, note that Theorem~\ref{thm:main} follows from Lemmas~\ref{lemma:splitlarge-1},~\ref{lemma:splitlarge-2} and~\ref{lemma:splitlarge-3}.

In Section~\ref{sub:sec:normal} we prove Lemmas~\ref{lemma:splitlarge-1} and~\ref{lemma:splitlarge-2}.
The starting point for proving these lemmas is a specific coloring of $E(K_n)$, which we call \emph{normal}, given in Definition~\ref{def:normal}.
In Section~\ref{sub:sec:strange} we prove Lemma~\ref{lemma:splitlarge-3}.
For proving this result, we start with an intricate coloring of $E(K_n)$, which we call \emph{strange} (see Definition~\ref{def:strange}).

Given a graph $G$, we say that the edge $uv \in E(G)$ is a \emph{conflicting edge} if $d_G(u) = d_G(v)$.

\subsection{Normal colorings of complete graphs}%
\label{sub:sec:normal}

We start this section by defining \emph{normal colorings} of complete graphs.
See Figure~\ref{fig:normal} for example.

\begin{definition}[\emph{Normal colorings}]%
\label{def:normal}
    Given a complete graph $G$ with $n$ vertices and a sequence $\vec V = (v_1, \ldots, v_n)$ of~$V(G)$, the \emph{normal coloring for $\vec V$} is the $2$-edge coloring $\phi \colon E(G) \to \{\red, \blue\}$ defined as follows, where $X_1 = \{v_1, \ldots, v_{\ceil{n/2}}\}$ and $X_2 = V(G) \setminus X_1$:
    \begin{enumerate}[label=\rmlabel]
        \item $G_\red[X_1]$ is a complete graph;
        \item $G_\red[X_2]$ contains no edges;
        \item $N_{G_\red}(v_i) = \{v_1, \ldots, v_{n-i+1}\}$ for $\ceil{n/2}+1 \leq i \leq n$.
    \end{enumerate}
\end{definition}

Note that in a normal coloring of a complete graph $G$ for a sequence $\vec V = (v_1, \ldots, v_n)$, we have 
\begin{itemize}
    \item $d_{G_{\red, \phi}}(v_i) = n - i \text{, for } 1 \leq i \leq \ceil{n/2}$;
    \item $d_{G_{\red, \phi}}(v_i) = n - i + 1 \text{, for } \ceil{n / 2} + 1 \leq i \leq n$.
\end{itemize}

Therefore, we know that, for a normal coloring $\phi$ of $G$,
\begin{equation}
\label{eq:normal:samedegree}
    \text{the only vertices with same degree in } G_{\red,\phi} \text{ (and\ } G_{\blue,\phi} \text{) are } v_{\ceil{n/2}} \tand v_{\ceil{n/2} + 1} \enspace.
\end{equation}

From the definition of normal colorings and~\eqref{eq:normal:samedegree}, since there is a (unique) conflicting edge $v_{\ceil{n/2}} v_{\ceil{n/2}+1}$, we know that if $n$ is even (resp.\ odd), then $G_\blue$ (resp.\ $G_\red$) is locally irregular and $G_\red$ (resp.\ $G_\blue$) is not locally irregular.

\begin{figure}
    \centering
\begin{tikzpicture}[scale=0.8]

    \pgfmathsetmacro{\qtd}{int(10)}
    \pgfmathsetmacro{\raio}{2.7}

    \tikzset{black vertex/.style={circle,draw,minimum size=2mm,inner sep=0pt,outer sep=1pt,fill=black, color=black}}
    \tikzset{redconf vertex/.style={circle,minimum size=3mm,inner sep=0pt,outer sep=1pt,fill=red,draw=white}}
    \tikzset{blueconf vertex/.style={circle,minimum size=3mm,inner sep=0pt,outer sep=1pt,fill=blue,draw=white}}

    \tikzstyle{red edge}=[line width=1, red]
  
    \pgfmathsetmacro{\espaco}{360/\qtd}
    \pgfmathsetmacro{\espacoMetade}{360/(\qtd*2)}    
    \pgfmathsetmacro{\rot}{-\espaco - 90 + \espacoMetade}
        
    \pgfmathsetmacro{\qtdMetade}{int(\qtd/2)}
    \pgfmathsetmacro{\qtdMetadeMenosDois}{int(\qtd/2 - 2)}
    \pgfmathsetmacro{\qtdMetadeMenosUm}{int(\qtd/2 -1)}
    \pgfmathsetmacro{\qtdMetadeMaisDois}{int(\qtd/2 + 2)}
    \pgfmathsetmacro{\qtdMetadeMaisUm}{int(\qtd/2 + 1)}
    \pgfmathsetmacro{\qtdMetadeMaisTres}{int(\qtd/2 + 3)}
    \pgfmathsetmacro{\qtdMetadeMaisQuatro}{int(\qtd/2 + 4)}

    \pgfmathsetmacro{\qtdMenosTres}{int(\qtd - 3)}
    \pgfmathsetmacro{\qtdMenosDois}{int(\qtd - 2)}
    \pgfmathsetmacro{\qtdMenosUm}{int(\qtd-1)}

    \draw[very thick,dashed,black] (90:\raio + 1) -- (-90:\raio + 1);

    \foreach \x in {1,...,\qtd} {
        \pgfmathsetmacro{\ang}{\x*\espaco+\rot}
        \node (\x) [black vertex] at (\ang:\raio) {};
    }

    \foreach \x in {1,...,\qtdMetade} {
        \pgfmathsetmacro{\xMaisUm}{int(\x + 1)}
        \pgfmathsetmacro{\fim}{int(\qtd - \x + 1)}
        \foreach \y in {\xMaisUm,...,\fim} {
            \draw[red edge] (\x) -- (\y);
        }
    }

    \foreach \x in {1,...,\qtdMetade} {
        \pgfmathsetmacro{\redDegree}{int(\qtd - \x)}
        \pgfmathsetmacro{\blueDegree}{int(\x - 1)}
        \node (\x_label) [ ] at (\x * \espaco+\rot:\raio + 1) {$v_{\x}$ ({\color{red} $\redDegree$}, {\color{blue} $\blueDegree$})};
    }

    \foreach \x in {\qtdMetadeMaisUm,...,\qtd} {
        \pgfmathsetmacro{\redDegree}{int(\qtd - \x + 1)}
        \pgfmathsetmacro{\blueDegree}{int(\x - 2)}
        \node (\x_label) [ ] at (\x * \espaco+\rot:\raio + 1) {$v_{\x}$ ({\color{red} $\redDegree$}, {\color{blue} $\blueDegree$})};
    }

    \node () [redconf vertex] at (\qtdMetade * \espaco+\rot:\raio) {};

    \node () [redconf vertex] at (\qtdMetadeMaisUm * \espaco+\rot:\raio) {};
  
\end{tikzpicture}
     \hfill
\begin{tikzpicture}[scale=0.8]

    \pgfmathsetmacro{\qtd}{int(11)}
    \pgfmathsetmacro{\raio}{2.7}

    \tikzset{black vertex/.style={circle,draw,minimum size=2mm,inner sep=0pt,outer sep=1pt,fill=black, color=black}}
    \tikzset{redconf vertex/.style={circle,minimum size=3mm,inner sep=0pt,outer sep=1pt,fill=red,draw=white}}
    \tikzset{blueconf vertex/.style={circle,minimum size=3mm,inner sep=0pt,outer sep=1pt,fill=blue,draw=white}}

    \tikzstyle{red edge}=[line width=1, red]
  
    \pgfmathsetmacro{\espaco}{360/\qtd}
    \pgfmathsetmacro{\espacoMetade}{360/(\qtd*2)}    
    \pgfmathsetmacro{\rot}{-\espaco - 90 + \espacoMetade}
        
    \pgfmathsetmacro{\qtdMetade}{int(\qtd/2)}
    \pgfmathsetmacro{\qtdMetadeMenosDois}{int(\qtd/2 - 2)}
    \pgfmathsetmacro{\qtdMetadeMenosUm}{int(\qtd/2 -1)}
    \pgfmathsetmacro{\qtdMetadeMaisDois}{int(\qtd/2 + 2)}
    \pgfmathsetmacro{\qtdMetadeMaisUm}{int(\qtd/2 + 1)}
    \pgfmathsetmacro{\qtdMetadeMaisTres}{int(\qtd/2 + 3)}
    \pgfmathsetmacro{\qtdMetadeMaisQuatro}{int(\qtd/2 + 4)}

    \pgfmathsetmacro{\qtdMenosTres}{int(\qtd - 3)}
    \pgfmathsetmacro{\qtdMenosDois}{int(\qtd - 2)}
    \pgfmathsetmacro{\qtdMenosUm}{int(\qtd-1)}

    \pgfmathsetmacro{\ang}{\qtdMetadeMaisUm*\espaco+\espacoMetade+\rot}
    \draw[very thick,dashed,black] (\ang:\raio + 1) -- (-90:\raio + 1);

    \foreach \x in {1,...,\qtd} {
        \pgfmathsetmacro{\ang}{\x*\espaco+\rot}
        \node (\x) [black vertex] at (\ang:\raio) {};
    }

    \foreach \x in {1,...,\qtdMetade} {
        \pgfmathsetmacro{\xMaisUm}{int(\x + 1)}
        \pgfmathsetmacro{\fim}{int(\qtd - \x + 1)}
        \foreach \y in {\xMaisUm,...,\fim} {
            \draw[red edge] (\x) -- (\y);
        }
    }

    \foreach \x in {1,...,\qtdMetadeMaisUm} {
        \pgfmathsetmacro{\redDegree}{int(\qtd - \x)}
        \pgfmathsetmacro{\blueDegree}{int(\x - 1)}
        \node (\x_label) [ ] at (\x * \espaco+\rot:\raio + 1) {$v_{\x}$ ({\color{red} $\redDegree$}, {\color{blue} $\blueDegree$})};
    }

    \foreach \x in {\qtdMetadeMaisDois,...,\qtd} {
        \pgfmathsetmacro{\redDegree}{int(\qtd - \x + 1)}
        \pgfmathsetmacro{\blueDegree}{int(\x - 2)}
        \node (\x_label) [ ] at (\x * \espaco+\rot:\raio + 1) {$v_{\x}$ ({\color{red} $\redDegree$}, {\color{blue} $\blueDegree$})};
    }

    \node () [blueconf vertex] at (\qtdMetadeMaisUm * \espaco+\rot:\raio) {};

    \node () [blueconf vertex] at (\qtdMetadeMaisDois * \espaco+\rot:\raio) {};
  
\end{tikzpicture}
 
    \caption{Normal colorings of $E(K_n)$ with sequence $(v_1, \ldots, v_n)$, for $n = 10, 11$.
            The vertices of the conflicting edge are highlighted.
            For better visualization, we omit all blue edges.}
    \label{fig:normal}
\end{figure}
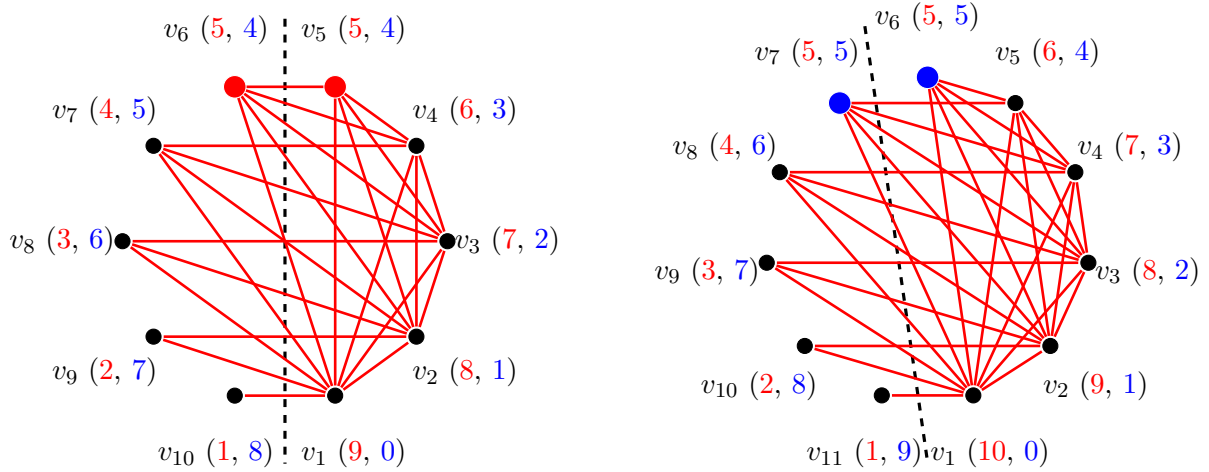%

The following proposition will be useful for proving Lemma~\ref{lemma:splitlarge-1}.

\begin{proposition}%
\label{prop:n-1-distinct-degrees}
    Let $G$ be a connected graph with $V(G) = \{u_1, u_2, \ldots, u_n\}$ and $d_G(u_1) \geq \cdots \geq d_G(u_n)$.
    If $G$ contains only one pair of vertices $u, v$ with $d_G(u) = d_G(v)$, then the following holds:
    \begin{enumerate}[label=\rmlabel]
        \item\label{it:n-1-distinct-degrees-0} $d_G(u) = d_G(v) = \floor{n/2}$;
        \item\label{it:n-1-distinct-degrees-1} $u = u_{\ceil{n/2}}$ and $v = u_{\ceil{n/2} + 1}$; 
        \item\label{it:n-1-distinct-degrees-2} $X = \{u_1, \ldots, u_{\ceil{n/2} - 1}\}$ is a clique and $Y = \{u_{\ceil{n/2} + 2}, \ldots, u_n\}$ is a stable set;
        \item\label{it:n-1-distinct-degrees-3} $X \subseteq N_G(u) \cap N_G(v)$;
        \item\label{it:n-1-distinct-degrees-4} $(N_G(u) \cup N_G(v)) \cap Y = \emptyset$;
        \item\label{it:n-1-distinct-degrees-5} $uv \in E(G)$ if and only if $n$ is even.
    \end{enumerate}
\end{proposition}
\begin{proof}
    The proof follows by induction on $n$.
    If $n = 2$, then $G \simeq K_2$, and if $n = 3$, then $G \simeq P_3$.
    In both cases,~\ref{it:n-1-distinct-degrees-0}-\ref{it:n-1-distinct-degrees-5} hold.
    Thus, we may assume that $n \geq 4$.

    Since $G$ is a connected graph with $n$ vertices and $u$ and $v$ are the only vertices of $G$ with the same degree, there are $n - 1$ distinct values of degrees in $G$.
    Moreover, since $G$ is connected, for any vertex $w$ of $G$ we have $1 \leq d_G(w) \leq n - 1$, and as a result of this, we know that the set of degrees of all vertices in $G$ is $\{1, 2, \ldots, n - 1\}$.
    Therefore, $d_G(u_1) = n - 1$ and $d_G(u_n) = 1$.
    Let $G' = G - \{u_1, u_n\}$.
    Note that $d_{G'}(w) = d_G(w) - 1$ for all $w \in V(G) \setminus \{u_1, u_n\}$.
    
    We will show that $G'$ is a connected graph with only one pair of vertices with the same degree.
    If $d_{G'}(u) \in \{1,n - 1\}$, then the vertices of $G'$ have distinct degrees.
    In particular, there exists a non-trivial component of $G'$ where all vertices have distinct degrees, which is an absurd.
    Thus, we may assume that $d_{G'}(u) \notin \{1,n - 1\}$, and hence $G'$ has precisely two vertices with the same degree.
    Now note that graph $G'$ has no trivial components, since $u_n$ is the only vertex of $G$ with degree $1$.
    Also, if $G'$ had more than one component, then it would contain a component where all vertices have distinct degrees, which is an absurd.
    Therefore, the graph $G'$ is connected and contains precisely two vertices with the same degree ($u$ and $v$), and hence, by induction hypothesis, \ref{it:n-1-distinct-degrees-0}-\ref{it:n-1-distinct-degrees-5} hold for $G'$.

    For clarity, let $V(G') = \{u'_1, u'_2, \ldots, u'_{n'}\} = \{u_2, u_3, \ldots, u_{n-1}\}$.
    Since~\ref{it:n-1-distinct-degrees-0} holds for $G'$, and $d_{G'}(w) = d_G(w) - 1$ for all $w \in V(G) \setminus \{u_1, u_n\}$, we have $d_G(u) = d_G(v) = d_{G'}(u)+1 = \floor{n'/2} = \floor{(n-2)/2} + 1 = \floor{n/2}$.
    Thus, \ref{it:n-1-distinct-degrees-0} holds for $G$.
    Since~\ref{it:n-1-distinct-degrees-1} holds for $G'$, the vertices $u = u'_{\ceil{n'/2}} = u_{\ceil{(n - 2)/2}+1} = u_{\ceil{n/2}}$ and $v = u'_{\ceil{n'/2}+1} = u_{(\ceil{(n - 2)/2} + 1)+1} = u_{\ceil{n/2} + 1}$ have the same degree in $G'$, and consequently in $G$, and hence~\ref{it:n-1-distinct-degrees-1} holds for $G$. 
    Since~\ref{it:n-1-distinct-degrees-2} holds for $G'$, the set $X' = \{u_2, \ldots, u_{\ceil{n / 2} - 1}\}$ is a clique of $G'$, the set $Y' = \{u_{\ceil{n/2} + 2}, \ldots, u_{n - 1}\}$ is a stable set of $G'$, and hence $X = \{u_1\} \cup X'$ is a clique of $G$ and $Y = \{u_n\} \cup Y'$ is a stable set of $G$, from where we conclude that~\ref{it:n-1-distinct-degrees-2} holds for $G$.
    Since~\ref{it:n-1-distinct-degrees-3} holds for $G'$, and $X' \subseteq N_{G'}(u) \cap N_{G'}(v)$, and since $d_G(u_1) = n-1$, we have $X \subseteq N_G(u)$ and $X \subseteq N_G(v)$.
    Therefore, \ref{it:n-1-distinct-degrees-3} holds for $G$. 
    Since~\ref{it:n-1-distinct-degrees-4} holds for $G'$, and $(N_{G'}(u) \cup N_{G'}(v)) \cap Y' = \emptyset$, and since $u_n$ has degree $1$ in $G$ and $u_1u_n \in E(G)$, we have $N_{G}(u) \cap Y = \emptyset$ and $N_{G}(v) \cap Y = \emptyset$.
    Therefore, \ref{it:n-1-distinct-degrees-4} holds for $G$. 
    Finally, since~\ref{it:n-1-distinct-degrees-5} holds for $G'$, and $G'$ has $n - 2$ vertices,~\ref{it:n-1-distinct-degrees-5} holds for $G$, which finishes the proof.
\end{proof}%

\begin{lemma}%
\label{lemma:splitlarge-1}
    Let $G(X,Y)$ be a split graph with $X = \{v_1, \ldots, v_n\}$ where $d_1 \geq \cdots \geq d_n$ and $n \geq 4$.
    If $d_1 < \floor{n/2}$ and $d_2 = 0$, then $\chii(G) = 3$.
\end{lemma}
\begin{proof}
    Let $G(X,Y)$ be a split graph with $X = \{v_1, \ldots, v_n\}$, $d_1 \geq \cdots \geq d_n$, $n \geq 4$, $d_1 < \floor{n/2}$, and $d_2 = 0$.
    We start by proving that $\chii(G) \geq 3$, and then we exhibit a coloring showing that $\chii(G) \leq 3$.

    \begin{claim}
    \label{claim:splitlarge-1:chimaior}
        $\chii(G) \geq 3$.
    \end{claim}
    \begin{claimproof}
        Since $d_1 \geq \cdots d_n \geq 0$, $n \geq 4$, and $d_2 = 0$, then by Fact~\ref{fact:splitlarge-0} we have $\chii(G) \geq 2$.
        Towards a contradiction, suppose that $\chii(G) = 2$, and let $\phi \colon E(G) \to \{\red,\blue\}$ be a locally irregular $2$-edge coloring of $G$.
        Let $H_\red = G_{\red,\phi}[X]$ and $H_\blue = G_{\blue,\phi}[X]$.

        Suppose that there exists a pair of vertices $v_x$ and $v_y$ such that $d_{H_\red}(v_x) = d_{H_\red}(v_y)$ and $2 \leq x < y$.
        Since $G_\red$ is locally irregular and
        \begin{equation*}
            d_{G_\red}(v_x) = d_{H_\red}(v_x) = d_{H_\red}(v_y) = d_{G_\red}(v_y) \enspace ,
        \end{equation*}
        $\phi(v_xv_y)$ is blue, but
        \begin{equation*}
            d_{G_\blue}(v_x) = n-1 - d_{G_\red}(v_x) = n-1 - d_{G_\red}(v_y) = d_{G_{\blue}}(v_y) \enspace,
        \end{equation*}
        a contradiction to the fact that $G_\blue$ is locally irregular.
        Therefore, for every pair of vertices~$v_x$ and~$v_y$ with $2 \leq x < y$, we have $d_{H_\red}(v_x) \neq d_{H_\red}(v_y)$.
        As a result, if $H_\red$ contains a pair of vertices of same degree, then it is unique and one of them must be $v_1$.
        By the Pigeonhole Principle, every connected graph with at least two vertices has at least one pair of vertices with the same degree.
        Thus $H_\red$ can have at most one trivial component and, since $n \geq 4$, it has precisely one non-trivial component.
        Similarly, $H_\blue$ has precisely one pair of vertices of same degree, one of these vertices being $v_1$, and it has precisely one non-trivial component and at most one trivial one.

        Let $v_x$ and $v_y$ be the vertices with the same degree as $v_1$ in the color red and blue, respectively.
        Note that if two vertices $v_w \neq v_1$ have the same degree in the color red, then they also have the same degree in the color blue and vice-versa, and since both $H_\red$ and $H_\blue$ have only one pair of vertices with the same degree, $v_x = v_y$.
        This means that the edge $v_1v_x$ is conflicting in $H_{\phi(v_1v_x)}$.
        Therefore, we must have $d_1 \geq 1$, as otherwise $G_\blue$ and $G_\red$ would not be locally irregular. 

        Suppose, without loss of generality, that $\phi(v_1v_x)$ is red.
        Let $K$ be the component of $H_\red$ containing the edge $v_1v_x$.
        Since $H_\red$ contains precisely one non-trivial component and at most one trivial component, $K$ has $n'$ vertices where $n' \geq n - 1$.
        By Proposition~\ref{prop:n-1-distinct-degrees}, we have that $d_K(v_1) = \floor{n'/2}$ and, since the edge $v_1v_x$ exists in $K$, we know that $n'$ is even.
        Moreover, if $X'$ is the set of vertices with degree at least $n'/2 + 1$ in $K$, then $X' \subset X$, $|X'| = n'/2 - 1$ and $X' \subseteq N_K(v_1)$ also by Proposition~\ref{prop:n-1-distinct-degrees}.
        Since 
        \begin{equation*}
            d_{G_\red}(v_1) = d_{H_\red}(v_1) + d_1 \enspace ,
        \end{equation*}
        on one hand we have
        \begin{equation*}
            d_{G_\red}(v_1) \geq n'/2 + 1 \enspace ,
        \end{equation*}
        and, on the other hand,
        \begin{equation*}
            d_{G_\red}(v_1) < n'/2 + \floor{n/2} = n' \enspace .
        \end{equation*}
        But since $d_{G_\red}(w) = d_K(w)$ for any $w \in X'$, this means that $v_1$ has the same degree in $G_\red$ as some vertex of $X'$, a contradiction to the fact that $G_\red$ is locally irregular.
    \end{claimproof}

    \begin{claim}
    \label{claim:splitlarge-1:chimenor}
        $\chii(G) \leq 3$.
    \end{claim}
    \begin{claimproof}
        Let $\phi$ be a normal coloring for the sequence $(v_2, \ldots, v_{\ceil{n/2}}, v_1, v_{\ceil{n/2}+1}, \ldots, v_n)$.
        Consider the coloring $\phi' \colon E(G) \to \{\red,\blue,\green\}$ defined as follows: $\phi'(v_1y) = \green$ for all $y \in Y$, $\phi'(v_1v_{\ceil{n/2}+1}) = \green$, and any other edge $e$ of $G$ has $\phi'(e) = \phi(e)$.
        If $d_1 = 0$, then we also do $\phi'(v_1v_{\ceil{n/2}}) = \green$ if $n$ is even or $\phi'(v_{\ceil{n/2}+1}v_{\ceil{n/2}+2}) = \green$ otherwise.
        See Figure~\ref{fig:normal-3} for examples of $\phi'$ with $n=8$ and $n=11$.

        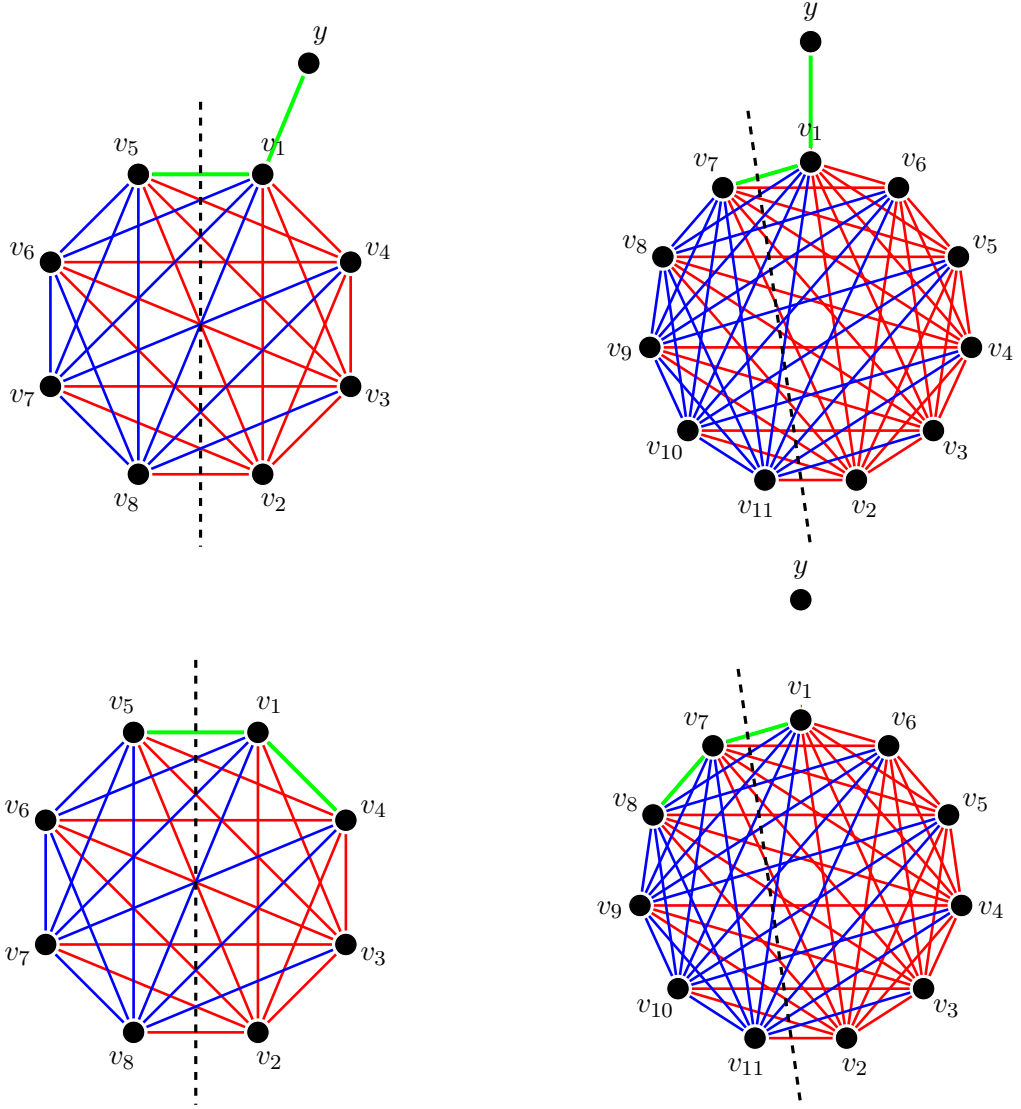
\begin{figure}
            \hfill
\begin{tikzpicture}[scale=0.8]

    \tikzset{black vertex/.style={circle,minimum size=3mm,inner sep=0pt,outer sep=1pt,fill=black,draw=white}}

    \pgfmathsetmacro{\qtd}{int(8)}
    \pgfmathsetmacro{\raio}{2.7}
    \pgfmathsetmacro{\espaco}{360/\qtd}
    \pgfmathsetmacro{\espacoMetade}{360/(\qtd*2)}    
    \pgfmathsetmacro{\rot}{-\espaco - 90 + \espacoMetade}
    \pgfmathsetmacro{\qtdMetade}{int(\qtd/2)}
    \pgfmathsetmacro{\qtdMetadeMaisUm}{int(\qtd/2 + 1)}
    \pgfmathsetmacro{\qtdMetadeMenosUm}{int(\qtd/2 - 1)}

    \tikzstyle{red edge}=[line width=1, red]
    \tikzstyle{blue edge}=[line width=1, blue]
    \tikzstyle{green edge}=[line width=1.5, green]

    \foreach \x in {1,...,\qtd} {
        \pgfmathsetmacro{\ang}{\x*\espaco+\rot}
        \node (\x) [black vertex] at (\ang:\raio) {};
    }
    \node (y) [black vertex] at (\qtdMetade*\espaco+\rot:\raio+2) {};

    \foreach \x in {1,...,\qtdMetade} {
        \pgfmathsetmacro{\ini}{int(\x + 1)}
        \pgfmathsetmacro{\fim}{int(\qtd - \ini + 2)}
        \foreach \y in {\ini,...,\fim} {
            \draw[red edge] (\x) -- (\y);
        }
    }

    \foreach \x in {\qtdMetadeMaisUm,...,\qtd} {
        \pgfmathsetmacro{\fim}{int(\x - 1)}
        \pgfmathsetmacro{\ini}{int(\qtd - \x + 2)}
        \foreach \y in {\ini,...,\fim}
            \draw[blue edge] (\x) -- (\y);
    }

    \draw[green edge] (\qtdMetade) -- (\qtdMetadeMaisUm);
    \draw[green edge] (\qtdMetade) -- (y);

    \foreach \x in {1,...,\qtdMetadeMenosUm} {
        \pgfmathsetmacro{\xMaisUm}{int(\x + 1)}
        \node (\x_label) [] at (\x*\espaco+\rot:\raio + 0.5) {$v_{\xMaisUm}$};
    }
    \node (1_label) [] at (\qtdMetade*\espaco+\rot:\raio + 0.5) {$v_1$};
    \foreach \x in {\qtdMetadeMaisUm,...,\qtd}
        \node (\x_label) [] at (\x*\espaco+\rot:\raio + 0.5) {$v_{\x}$};
    \node (y_label) [] at (\qtdMetade*\espaco+\rot:\raio+2.5) {$y$};

    \pgfmathsetmacro{\ang}{\qtdMetade*\espaco+\espacoMetade+\rot}
    \draw[very thick,dashed,black] (\ang:\raio + 1) -- (-90:\raio + 1);        
\end{tikzpicture}
             \hfill
\begin{tikzpicture}[scale=0.8]

    \tikzset{black vertex/.style={circle,minimum size=3mm,inner sep=0pt,outer sep=1pt,fill=black,draw=white}}

    \pgfmathsetmacro{\qtd}{int(11)}
    \pgfmathsetmacro{\raio}{2.7}
    \pgfmathsetmacro{\espaco}{360/\qtd}
    \pgfmathsetmacro{\espacoMetade}{360/(\qtd*2)}    
    \pgfmathsetmacro{\rot}{-\espaco - 90 + \espacoMetade}
    \pgfmathsetmacro{\qtdMetade}{int(\qtd/2)}
    \pgfmathsetmacro{\qtdMetadeMaisUm}{int(\qtd/2 + 1)}
    \pgfmathsetmacro{\qtdMetadeMaisDois}{int(\qtd/2 + 2)}
    \pgfmathsetmacro{\qtdMetadeMenosUm}{int(\qtd/2 - 1)}

    \tikzstyle{red edge}=[line width=1, red]
    \tikzstyle{blue edge}=[line width=1, blue]
    \tikzstyle{green edge}=[line width=1.5, green]

    \foreach \x in {1,...,\qtd} {
        \pgfmathsetmacro{\ang}{\x*\espaco+\rot}
        \node (\x) [black vertex] at (\ang:\raio) {};
    }
    \node (y) [black vertex] at (\qtdMetadeMaisUm*\espaco+\rot:\raio+2) {};

    \foreach \x in {1,...,\qtdMetadeMaisUm} {
        \pgfmathsetmacro{\ini}{int(\x + 1)}
        \pgfmathsetmacro{\fim}{int(\qtd - \ini + 2)}
        \foreach \y in {\ini,...,\fim} {
            \draw[red edge] (\x) -- (\y);
        }
    }

    \foreach \x in {\qtdMetadeMaisDois,...,\qtd} {
        \pgfmathsetmacro{\fim}{int(\x - 1)}
        \pgfmathsetmacro{\ini}{int(\qtd - \x + 2)}
        \foreach \y in {\ini,...,\fim}
            \draw[blue edge] (\x) -- (\y);
    }

    \draw[green edge] (\qtdMetadeMaisUm) -- (\qtdMetadeMaisDois);
    \draw[green edge] (\qtdMetadeMaisUm) -- (y);

    \foreach \x in {1,...,\qtdMetade} {
        \pgfmathsetmacro{\xMaisUm}{int(\x + 1)}
        \node (\x_label) [] at (\x*\espaco+\rot:\raio + 0.5) {$v_{\xMaisUm}$};
    }
    \node (1_label) [] at (\qtdMetadeMaisUm*\espaco+\rot:\raio + 0.5) {$v_1$};
    \foreach \x in {\qtdMetadeMaisDois,...,\qtd}
        \node (\x_label) [] at (\x*\espaco+\rot:\raio + 0.5) {$v_{\x}$};
    \node (y_label) [] at (\qtdMetadeMaisUm*\espaco+\rot:\raio+2.5) {$y$};

    \pgfmathsetmacro{\ang}{\qtdMetadeMaisUm*\espaco+\espacoMetade+\rot}
    \draw[very thick,dashed,black] (\ang:\raio + 1) -- (-90:\raio + 1);        
\end{tikzpicture}
             \hfill
            \\
            \hfill
\begin{tikzpicture}[scale=0.8]

    \tikzset{black vertex/.style={circle,minimum size=3mm,inner sep=0pt,outer sep=1pt,fill=black,draw=white}}

    \pgfmathsetmacro{\qtd}{int(8)}
    \pgfmathsetmacro{\raio}{2.7}
    \pgfmathsetmacro{\espaco}{360/\qtd}
    \pgfmathsetmacro{\espacoMetade}{360/(\qtd*2)}    
    \pgfmathsetmacro{\rot}{-\espaco - 90 + \espacoMetade}
    \pgfmathsetmacro{\qtdMetade}{int(\qtd/2)}
    \pgfmathsetmacro{\qtdMetadeMaisUm}{int(\qtd/2 + 1)}
    \pgfmathsetmacro{\qtdMetadeMenosUm}{int(\qtd/2 - 1)}

    \tikzstyle{red edge}=[line width=1, red]
    \tikzstyle{blue edge}=[line width=1, blue]
    \tikzstyle{green edge}=[line width=1.5, green]

    \foreach \x in {1,...,\qtd} {
        \pgfmathsetmacro{\ang}{\x*\espaco+\rot}
        \node (\x) [black vertex] at (\ang:\raio) {};
    }

    \foreach \x in {1,...,\qtdMetade} {
        \pgfmathsetmacro{\ini}{int(\x + 1)}
        \pgfmathsetmacro{\fim}{int(\qtd - \ini + 2)}
        \foreach \y in {\ini,...,\fim} {
            \draw[red edge] (\x) -- (\y);
        }
    }

    \foreach \x in {\qtdMetadeMaisUm,...,\qtd} {
        \pgfmathsetmacro{\fim}{int(\x - 1)}
        \pgfmathsetmacro{\ini}{int(\qtd - \x + 2)}
        \foreach \y in {\ini,...,\fim}
            \draw[blue edge] (\x) -- (\y);
    }

    \draw[green edge] (\qtdMetade) -- (\qtdMetadeMaisUm);
    \draw[green edge] (\qtdMetade) -- (\qtdMetadeMenosUm);

    \foreach \x in {1,...,\qtdMetadeMenosUm} {
        \pgfmathsetmacro{\xMaisUm}{int(\x + 1)}
        \node (\x_label) [] at (\x*\espaco+\rot:\raio + 0.5) {$v_{\xMaisUm}$};
    }
    \node (1_label) [] at (\qtdMetade*\espaco+\rot:\raio + 0.5) {$v_1$};
    \foreach \x in {\qtdMetadeMaisUm,...,\qtd}
        \node (\x_label) [] at (\x*\espaco+\rot:\raio + 0.5) {$v_{\x}$};

    \pgfmathsetmacro{\ang}{\qtdMetade*\espaco+\espacoMetade+\rot}
    \draw[very thick,dashed,black] (\ang:\raio + 1) -- (-90:\raio + 1);        
\end{tikzpicture}
             \hfill
\begin{tikzpicture}[scale=0.8]

    \tikzset{black vertex/.style={circle,minimum size=3mm,inner sep=0pt,outer sep=1pt,fill=black,draw=white}}

    \pgfmathsetmacro{\qtd}{int(11)}
    \pgfmathsetmacro{\raio}{2.7}
    \pgfmathsetmacro{\espaco}{360/\qtd}
    \pgfmathsetmacro{\espacoMetade}{360/(\qtd*2)}    
    \pgfmathsetmacro{\rot}{-\espaco - 90 + \espacoMetade}
    \pgfmathsetmacro{\qtdMetade}{int(\qtd/2)}
    \pgfmathsetmacro{\qtdMetadeMaisUm}{int(\qtd/2 + 1)}
    \pgfmathsetmacro{\qtdMetadeMaisDois}{int(\qtd/2 + 2)}
    \pgfmathsetmacro{\qtdMetadeMaisTres}{int(\qtd/2 + 3)}
    \pgfmathsetmacro{\qtdMetadeMenosUm}{int(\qtd/2 - 1)}

    \tikzstyle{red edge}=[line width=1, red]
    \tikzstyle{blue edge}=[line width=1, blue]
    \tikzstyle{green edge}=[line width=1.5, green]

    \foreach \x in {1,...,\qtd} {
        \pgfmathsetmacro{\ang}{\x*\espaco+\rot}
        \node (\x) [black vertex] at (\ang:\raio) {};
    }
    \node (y) [black vertex] at (\qtdMetadeMaisUm*\espaco+\rot:\raio+2) {};

    \foreach \x in {1,...,\qtdMetadeMaisUm} {
        \pgfmathsetmacro{\ini}{int(\x + 1)}
        \pgfmathsetmacro{\fim}{int(\qtd - \ini + 2)}
        \foreach \y in {\ini,...,\fim} {
            \draw[red edge] (\x) -- (\y);
        }
    }

    \foreach \x in {\qtdMetadeMaisDois,...,\qtd} {
        \pgfmathsetmacro{\fim}{int(\x - 1)}
        \pgfmathsetmacro{\ini}{int(\qtd - \x + 2)}
        \foreach \y in {\ini,...,\fim}
            \draw[blue edge] (\x) -- (\y);
    }

    \draw[green edge] (\qtdMetadeMaisUm) -- (\qtdMetadeMaisDois);
    \draw[green edge] (\qtdMetadeMaisDois) -- (\qtdMetadeMaisTres);

    \foreach \x in {1,...,\qtdMetade} {
        \pgfmathsetmacro{\xMaisUm}{int(\x + 1)}
        \node (\x_label) [] at (\x*\espaco+\rot:\raio + 0.5) {$v_{\xMaisUm}$};
    }
    \node (1_label) [] at (\qtdMetadeMaisUm*\espaco+\rot:\raio + 0.5) {$v_1$};
    \foreach \x in {\qtdMetadeMaisDois,...,\qtd}
        \node (\x_label) [] at (\x*\espaco+\rot:\raio + 0.5) {$v_{\x}$};
    \node (y_label) [] at (\qtdMetadeMaisUm*\espaco+\rot:\raio+2.5) {$y$};

    \pgfmathsetmacro{\ang}{\qtdMetadeMaisUm*\espaco+\espacoMetade+\rot}
    \draw[very thick,dashed,black] (\ang:\raio + 1) -- (-90:\raio + 1);        
\end{tikzpicture}
             \hfill
            \caption{Coloring $\phi'$ described in the proof of Claim~\ref{claim:splitlarge-1:chimenor} for $n=8,11$ when $d_1 = 0$ and $d_1 \neq 0$.}
            \label{fig:normal-3}
        \end{figure}

        Recall that the edge $v_1 v_{\ceil{n/2}+1}$ is the only one conflicting in $G_{\red,\phi}$ if $n$ is even, or in $G_{\blue,\phi}$ otherwise.
        In $\phi'$, such edge has color green and clearly there is no conflicting edge in $G_{\green,\phi}$.
        
        The degrees of $v_1$ and $v_{\ceil{n/2}+1}$ have decreased by one from $G_{\gamma,\phi}$ to $G_{\gamma,\phi'}$, for $\gamma = \phi(v_1v_{\ceil{n/2}+1})$.
        If $\gamma = \red$, then $v_{\ceil{n/2}+2}$ is the only vertex of $G$ with the same degree as $v_1$ and $v_{\ceil{n/2}+1}$ in $G_{\red,\phi'}$.
        However, $v_{\ceil{n/2}+2}$ is not a neighbor of $v_1$ or $v_{\ceil{n/2}+1}$ in $G_{\red,\phi'}$.
        Otherwise, if $\gamma = \blue$, then $v_{\ceil{n/2}}$ is the only vertex of $G$ with the same degree as $v_1$ and $v_{\ceil{n/2}+1}$ in $G_{\blue,\phi'}$.
        Likewise, $v_{\ceil{n/2}}$ is not a neighbor of $v_1$ or $v_{\ceil{n/2}+1}$ in $G_{\blue,\phi'}$.

        If $d_1 = 0$ and $n$ is even, then the degree of $v_{\ceil{n/2}}$ has also decreased by one from $G_{\red,\phi}$ to $G_{\red,\phi'}$.
        However, $d_{G_{\red,\phi'}}(v_{\ceil{n/2}}) = \ceil{n/2} = d_{G_{\red,\phi}}(v_1)$ and so it is the only vertex with such degree in $G_{\red,\phi'}$.
        If $d_1 = 0$ and $n$ is odd, then the degree of $v_{\ceil{n/2}+2}$ has decreased by one from $G_{\blue,\phi}$ to $G_{\blue,\phi'}$.
        Similarly, $d_{G_{\blue,\phi'}}(v_{\ceil{n/2}+2}) = \ceil{n/2}-1 = d_{G_{\blue,\phi}}(v_1)$ and so it is the only vertex with such degree in $G_{\blue,\phi'}$.

        Therefore, we conclude that $\phi'$ is a locally irregular 3-edge coloring of $G$, which implies $\chii(G)\leq 3$.
    \end{claimproof}

    Claims~\ref{claim:splitlarge-1:chimaior} and~\ref{claim:splitlarge-1:chimenor} conclude the proof of Lemma~\ref{lemma:splitlarge-1}.
\end{proof}%

In the proof of Lemma~\ref{lemma:splitlarge-2} below the reader may find it useful to refer to Figure~\ref{fig:normal}.

\begin{lemma}%
\label{lemma:splitlarge-2}
    Let $G(X,Y)$ be a split graph with $X = \{v_1, \ldots, v_n\}$ where $d_1 \geq \cdots \geq d_n$ and $d_{\floor{n/2}} \geq 1$.
    If $n \geq 3$, then, $\chii(G) \leq 2$.
\end{lemma}
\begin{proof}
    There are two cases to consider depending on the parity of $n$, but the only difference in the proofs is the coloring we give to $E(G[X])$, which are symmetric.
    For $n$ even, we start with a normal coloring $\phi \colon E(G[X]) \to \{\red,\blue\}$ for the sequence $(v_1, \ldots, v_{n/2}, v_n, v_{n-1}, \ldots, v_{n/2 + 1})$.
    In case $n$ is odd we consider a normal coloring of $E(G[X])$ for sequence $(v_{\ceil{n/2}}, \ldots, v_n, v_{\floor{n/2}}, \ldots, v_1)$.
    Thus, for the rest of this proof, we may assume, without loss of generality, that $n$ is even.
    
    Let $X_1 = \{v_1, v_2, \ldots, v_{n/2}\}$ and $X_2 = X \setminus X_1$.
    From~\eqref{eq:normal:samedegree} we know that the only vertices with the same degree in $G_\red[X]$ or $G_\blue[X]$ are $v_{n/2}$ and~$v_n$.

    We will obtain a locally irregular $2$-edge coloring of $G$ from $\phi$.
    We start by extending $\phi$ to a coloring $\phi'$ of $E(G)$ with colors red and blue in the following way.
    For all edges $xy$ between $X_1$ and $Y$ let $\phi'(xy) = \red$, and for all edges $xy$ between $X_2$ and~$Y$ let $\phi'(xy) = \blue$.
    
    Let us first analyze the graph $G_{\red,\phi'}$.
    Since $d_{n/2} \geq 1$, for every vertex $x \in X_1$ we have $d_{G_{\red,\phi'}}(x) > d_{G_{\red,\phi}}(x) \geq n/2$, and since $d_1 \geq \cdots \geq d_{n/2}$, the degree of any two vertices of $X_1$ remain different in $G_{\red,\phi'}$.
    Also, since there are no red edges between $v_n$ and $Y$, we have $d_{G_{\red,\phi'}}(v_n) = n/2 < d_{G_{\red,\phi'}}(x)$ for every $x \in X_1$.
    The red degree of vertices $v_{n/2+1}, \ldots, v_n$ are the same in $\phi$ and $\phi'$, and the red degree of any vertex $y \in Y$ is at most $n/2$, since there are no red edges between $X_2$ and $Y$.    
    Therefore, since the degrees of the vertices of $X_1$ in $G_{\red,\phi'}$ are at least $n/2 + 1$, we conclude that $G_{\red,\phi'}$ is locally irregular.
    
    It remains to show that $G_{\blue,\phi'}$ is locally irregular.
    Since $\phi$ is a normal coloring and $n$ is even, we know that $G_{\blue,\phi'}[X]$ is locally irregular.
    If there is no $y \in Y$ that has a neighbor $x \in X_2$ with $d_{G_{\blue,\phi'}}(y) = d_{G_{\blue,\phi'}}(x)$, then the result follows.
    Thus we may assume the opposite, i.e., there exist $y \in Y$ and $x \in X_2$ with the same degree in $G_{\blue,\phi'}$.
    Since the maximum possible degree of a vertex of $Y$ in $G_{\blue,\phi'}$ is $n/2$ and the minimum degree of a vertex of $X_2$ in $G_{\blue,\phi'}[X]$ is $n/2 - 1$, we conclude that 
    \begin{equation*}    
        d_{G_{\blue,\phi'}}(y) = d_{G_{\blue,\phi'}}(v_n) = n/2.
    \end{equation*}
    Therefore, because of the pair $y,v_n$, the graph $G_{\blue,\phi'}$ is not locally irregular.
    In this case, we can change the color of one or two edges in $\phi'$ to obtain a locally irregular $2$-edge coloring $\phi''$ for $G$, as we explain next.
    
    If $d_{n/2} \geq 2$, then let $\phi''$ be the coloring obtained from $\phi'$ by changing the color of $yv_{n}$ from blue to red.
    We claim that the graphs $G_{\red,\phi''}$ and $G_{\blue,\phi''}$ are locally irregular.
    In fact, this holds since $v_n$ has degree $n/2+1$ in $G_{\red,\phi''}$ and every vertex in $X_1$ has degree at least $n/2+2$ in~$G_{\red,\phi''}$.
    
    Now assume that $d_{n/2}= 1$ and let $z \in Y$ be the only neighbor of $v_{n/2}$ in $Y$.
    In this case consider the coloring $\phi''$ obtained from $\phi'$ by changing the color of $yv_{n}$ from blue to red and the color of $zv_{n/2}$ from red to blue.
    Although $d_{G_{\blue,\phi''}}(y) = d_{G_{\blue,\phi''}}(v_n) = (n/2) - 1$, they are not neighbors in $G_{\blue,\phi''}$.
    Also, any vertex in $X_2 \setminus \{v_n\}$ has degree at least $(n/2)+1$ in $G_{\blue,\phi''}$, so there are no conflicts in $G_{\blue,\phi''}$ involving $y$ or $v_n$.
    Note that we have $d_{G_{\red,\phi''}}(v_n) = (n/2) + 1$, but since $d_{G_{\red,\phi''}}(v_{n/2}) = n/2$ and every vertex in $X_1 \setminus \{v_{n/2}\}$ has red degree at least $(n/2) + 2$ in $G_{\red,\phi''}$, there are no conflicts in $G_{\red,\phi''}$ involving $v_n$. 
    This also implies that, since $d_{G_{\red,\phi''}}(v_{n/2}) = n/2$, there are no conflicts involving $v_{n/2}$ in $G_{\red,\phi''}$.
    Furthermore, since $d_{G_{\blue,\phi''}}(v_{n/2}) = n/2$ and any vertex in $X_2\setminus\{v_n\}$ has degree at least $(n/2)+1$ in $G_{\blue,\phi''}$, we conclude that there are no conflicts involving $v_{n/2}$ in $G_{\blue,\phi''}$.
    Therefore, $G_{\red,\phi''}$ and $G_{\blue,\phi''}$ are locally irregular, and the result follows.
\end{proof}%

\subsection{Strange colorings of complete graphs}%
\label{sub:sec:strange}

As in Section~\ref{sec:large-split}, we start by defining the colorings of complete graphs that are the starting point for proving the results in this section.
The following definition is technical, so we refer the reader to Figure~\ref{fig:strange} for a better understanding of it.

\begin{definition}[Strange coloring]%
\label{def:strange}
    Given a complete graph $G$ with $n$ vertices and a sequence $\vec V = (v_1, \ldots, v_n)$ of~$V(G)$, first consider a coloring $\phi' \colon E(G) \to \{\red, \blue\}$ defined as follows, where $X_1 = \{v_1, \ldots, v_{\ceil{n/2}}\}$ and $X_2 = V(G) \setminus X_1$:
    \begin{enumerate}[label=\rmlabel]
        \item $G_\red[X_1]$ is a complete graph;
        \item $G_\red[X_2]$ contains no edges;
        \item $N_{G_{\red,\phi'}}(v_i) = \{v_1, \ldots, v_{n - i}\}$ for $\ceil{n/2}+1 \leq i \leq n-1$;
        \item $\phi'(v_1v_n) = \red$;
        \item $\phi'(v_{\ceil{n/2}+1} v_{\floor{n/2}}) = \red$;
        \item All other edges are blue.
    \end{enumerate}
    The \emph{strange coloring} of $G$ for $\vec V$ is the coloring $\phi$ obtained from $\phi'$ by changing the color of the following edges, which we call \emph{strange edges}:
    \begin{itemize}
        \item $v_{\floor{n/2}}v_{\floor{n/2} - 1}$ becomes blue;
        \item $v_{\floor{n/2} - 1}v_{n-1}$ becomes red;
        \item $v_1v_{\ceil{n/2}+1}$ for $\ceil{n/2}$ even becomes blue;
        \item $v_1v_{\floor{n/2}+1}$ for $\ceil{n/2}$ odd becomes blue;
        \item $v_{n-2}v_{n-3}$, $v_{n-4}v_{n-5},$ $\ldots,$ $v_{n/2+4}v_{n/2+3}$ for $n = 0\text{ (mod 4)}$ become red;
        \item $v_{n-1}v_{\floor{n/2}-1}$, $v_{n-2}v_{\floor{n/2}}$, $v_{n-3}v_{\floor{n/2}+1}$, $v_{n-4}v_{n-5}$, $v_{n-6}v_{n-7},$ $\ldots,$ $v_{\floor{n/2}+3}v_{\floor{n/2}+2}$ for $n = 1\text{ (mod 4)}$ become red;
        \item $v_{n-2}v_{n-3}$, $v_{n-4}v_{n-5},$ $\ldots,$ $v_{n/2+3}v_{n/2+2}$ for $n = 2\text{ (mod 4)}$ become red;
        \item $v_{n-2}v_{n-3}$, $v_{n-4}v_{n-5},$ $\ldots,$ $v_{\floor{n/2}+4}v_{\floor{n/2}+3}$ for $n = 3\text{ (mod 4)}$ become red.
   \end{itemize}
\end{definition}

Note that in a strange coloring of a complete graph $G$ for a sequence $\vec V = (v_1, \ldots, v_n)$, we have
\begin{equation*}
    d_{G_{\red, \phi}}(v_i) = n - i + 1 \text{, for } 3 \leq i \leq n \enspace ,
\end{equation*}
\begin{equation*}
    d_{G_{\red, \phi}}(v_1) = n - \floor{n/2} - 1 \enspace ,
\end{equation*}
and
\begin{equation*}
    d_{G_{\red, \phi}}(v_2) = 
    \begin{cases}
        n - \floor{n/2} - 2 & \text{ if } \ceil{n/2} \text{ is even } \\
        n - \floor{n/2} - 1 & \text{ otherwise } \enspace . \\
    \end{cases}
\end{equation*}

Therefore, we know that, for a strange coloring $\phi$ of $G$,
\begin{equation}
\label{eq:strange:samereddegree}
    \text{the only vertices with same degree in } G_{\red,\phi} \text{ are } v_1 \tand v_{\floor{n/2} + 2} \enspace ,
\end{equation}
and
\begin{equation}
\label{eq:strange:samebluedegree}
    \text{the only vertices with same degree in } G_{\blue,\phi} \text{ are }
    \begin{cases}
        v_2 \tand v_{\floor{n/2} + 3} & \text{ if } \ceil{n/2} \text{ is even } \\
        v_2, v_1, \tand v_{\floor{n/2} + 1} & \text{ otherwise } \enspace .
    \end{cases}
\end{equation}

From the definition of strange coloring and by~\eqref{eq:strange:samereddegree} and~\eqref{eq:strange:samebluedegree}, we conclude that $G_\red$ has exactly one conflicting edge $v_1 v_{\floor{n/2}+2}$ while  
\begin{equation*}
    \text{ $G_\blue$ has exactly }
    \begin{cases}
        \text{one conflicting edge } v_2 v_{\floor{n/2}+3} & \text{ if } \ceil{n/2} \text{ is even }\\
        \text{two conflicting edges } v_2 v_{\floor{n/2}+2} \tand v_2 v_1 & \text{ otherwise } \enspace .
    \end{cases}
\end{equation*}

\begin{figure}[ht]
\begin{tikzpicture}[scale=0.8]

    \pgfmathsetmacro{\qtd}{int(10)}
    \pgfmathsetmacro{\raio}{2.7}

    \tikzset{black vertex/.style={circle,draw,minimum size=2mm,inner sep=0pt,outer sep=1pt,fill=black, color=black}}
    \tikzset{redconf vertex/.style={circle,minimum size=3mm,inner sep=0pt,outer sep=1pt,fill=red,draw=white}}
    \tikzset{blueconf vertex/.style={circle,minimum size=3mm,inner sep=0pt,outer sep=1pt,fill=blue,draw=white}}
    \tikzset{redblueconf vertex/.style={circle,minimum size=1mm,inner sep=0pt,outer sep=1pt,fill=blue,draw=blue}}

    \tikzstyle{blue edge}=[line width=4, blue]
    \tikzstyle{red edge}=[line width=1, red]
    \tikzstyle{strong red edge}=[line width=4, red]

    \pgfmathsetmacro{\espaco}{360/\qtd}
    \pgfmathsetmacro{\espacoMetade}{360/(\qtd*2)}    
    \pgfmathsetmacro{\rot}{-\espaco - 90 + \espacoMetade}
        
    \pgfmathsetmacro{\qtdMetade}{int(\qtd/2)}
    \pgfmathsetmacro{\qtdMetadeMenosDois}{int(\qtd/2 - 2)}
    \pgfmathsetmacro{\qtdMetadeMenosUm}{int(\qtd/2 -1)}
    \pgfmathsetmacro{\qtdMetadeMaisDois}{int(\qtd/2 + 2)}
    \pgfmathsetmacro{\qtdMetadeMaisUm}{int(\qtd/2 + 1)}
    \pgfmathsetmacro{\qtdMetadeMaisTres}{int(\qtd/2 + 3)}

    \pgfmathsetmacro{\qtdMenosTres}{int(\qtd - 3)}
    \pgfmathsetmacro{\qtdMenosDois}{int(\qtd - 2)}
    \pgfmathsetmacro{\qtdMenosUm}{int(\qtd-1)}

    \draw[very thick,dashed,black] (90:\raio + 1) -- (-90:\raio + 1);

    \foreach \x in {1,...,\qtd} {
        \pgfmathsetmacro{\ang}{\x*\espaco+\rot}
        \node (\x) [black vertex] at (\ang:\raio) {};
    }

    \foreach \x in {2,...,\qtd} {
        \pgfmathsetmacro {\y}{int(\qtd/2 + 1)}
        \ifthenelse{ \equal{\x}{\y} } {
            \draw[blue edge] (1) -- (\x);
        }{
            \draw[red edge] (1) -- (\x);
        }
    }

    \foreach \x in {2,...,\qtdMetadeMenosDois} {
        \pgfmathsetmacro{\xMaisUm}{int(\x + 1)}
        \pgfmathsetmacro{\fim}{int(\qtd - \x)}
        \foreach \y in {\xMaisUm,...,\fim} {
            \draw[red edge] (\x) -- (\y);
        }
    }

    \foreach \x in {2,...,\qtdMetadeMenosDois} {
        \pgfmathsetmacro{\xMaisUm}{int(\x + 1)}
        \pgfmathsetmacro{\fim}{int(\qtd - \x)}
        \foreach \y in {\xMaisUm,...,\fim} {
            \draw[red edge] (\x) -- (\y);
        }
    }

    \draw[blue edge] (\qtdMetadeMenosUm) -- (\qtdMetade);
    \draw[red edge] (\qtdMetadeMenosUm) -- (\qtdMetadeMaisUm);
    \draw[strong red edge] (\qtdMetadeMenosUm) -- (\qtdMenosUm);

    \draw[red edge] (\qtdMetade) -- (\qtdMetadeMaisUm);

    \pgfmathsetmacro{\qtdMetadeMaisDoisMaisStep}{int(\qtdMetadeMaisDois + 2)}
    \ifthenelse{\qtdMetadeMaisDoisMaisStep < \qtdMenosTres \OR \qtdMetadeMaisDoisMaisStep = \qtdMenosTres}{
        \foreach \x in {\qtdMetadeMaisDois,\qtdMetadeMaisDoisMaisStep,...,\qtdMenosTres} {
            \pgfmathsetmacro{\xMaisUm}{int(\x + 1)}
            \draw[strong red edge] (\x) -- (\xMaisUm);
        }
    }{
        \pgfmathsetmacro{\x}{int(\qtdMetadeMaisDois)}
        \pgfmathsetmacro{\xMaisUm}{int(\x + 1)}
        \draw[strong red edge] (\x) -- (\xMaisUm);
    }

    \pgfmathsetmacro{\blueDegree}{int(\qtdMetade)}
    \pgfmathsetmacro{\redDegree}{int(\qtd - 1 - \blueDegree)}
  
    \node (1_label)              [ ] at (1 * \espaco+\rot:\raio + 1)                  {$v_{3}$ ({\color{red} $\qtdMenosDois$}, {\color{blue} $1$})};
    \node (metade_label)         [ ] at (\qtdMetade * \espaco+\rot:\raio + 1)         {$v_{1}$ ({\color{red} $\redDegree$}, {\color{blue} $\blueDegree$})};
    \node (metadeMaisUm_label)   [ ] at (\qtdMetadeMaisUm * \espaco+\rot:\raio + 1)   {$v_{\qtdMetadeMaisDois}$ ({\color{red} $\redDegree$}, {\color{blue} $\blueDegree$})};
    \node (metadeMaisDois_label) [ ] at (\qtdMetadeMaisDois * \espaco+\rot:\raio + 1) {$v_{2}$ ({\color{red} $\redDegree$}, {\color{blue} $\blueDegree$})};

    \foreach \x in {2,...,\qtdMetadeMenosUm} {
        \pgfmathsetmacro{\xMaisDois}{int(\x + 2)}
        \pgfmathsetmacro{\redDegree}{int(\qtd - \x - 1)}
        \pgfmathsetmacro{\blueDegree}{int(\x)}
        \node (\x_label) [ ] at (\x * \espaco+\rot:\raio + 1) {$v_{\xMaisDois}$ ({\color{red} $\redDegree$}, {\color{blue} $\blueDegree$})};
    }

    \foreach \x in {\qtdMetadeMaisTres,...,\qtd} {
        \pgfmathsetmacro{\redDegree}{int(\qtd - \x + 1)}
        \pgfmathsetmacro{\blueDegree}{int(\qtd - 1 - \redDegree)}
        \node (\x_label) [ ] at (\x * \espaco+\rot:\raio + 1) {$v_{\x}$ ({\color{red} $\redDegree$}, {\color{blue} $\blueDegree$})};
    }

    \node () [redconf vertex] at (\qtdMetade * \espaco+\rot:\raio) {};
    \node () [redblueconf vertex] at (\qtdMetade * \espaco+\rot:\raio) {};

    \node () [redconf vertex] at (\qtdMetadeMaisUm * \espaco+\rot:\raio) {};

    \node () [blueconf vertex] at (\qtdMetadeMaisDois * \espaco+\rot:\raio) {};
  
\end{tikzpicture}
     \hfill
\begin{tikzpicture}[scale=0.8]

    \pgfmathsetmacro{\qtd}{int(11)}
    \pgfmathsetmacro{\raio}{2.7}

    \tikzset{black vertex/.style={circle,draw,minimum size=2mm,inner sep=0pt,outer sep=1pt,fill=black, color=black}}
    \tikzset{redconf vertex/.style={circle,minimum size=3mm,inner sep=0pt,outer sep=1pt,fill=red,draw=white}}
    \tikzset{blueconf vertex/.style={circle,minimum size=3mm,inner sep=0pt,outer sep=1pt,fill=blue,draw=white}}
    \tikzset{redblueconf vertex/.style={circle,minimum size=1mm,inner sep=0pt,outer sep=1pt,fill=blue,draw=blue}}

    \tikzstyle{blue edge}=[line width=4, blue]
    \tikzstyle{red edge}=[line width=1, red]
    \tikzstyle{strong red edge}=[line width=4, red]

    \pgfmathsetmacro{\espaco}{360/\qtd}
    \pgfmathsetmacro{\espacoMetade}{360/(\qtd*2)}    
    \pgfmathsetmacro{\rot}{-\espaco - 90 + \espacoMetade}
        
    \pgfmathsetmacro{\qtdMetade}{int(\qtd/2)}
    \pgfmathsetmacro{\qtdMetadeMenosDois}{int(\qtd/2 - 2)}
    \pgfmathsetmacro{\qtdMetadeMenosUm}{int(\qtd/2 -1)}
    \pgfmathsetmacro{\qtdMetadeMaisDois}{int(\qtd/2 + 2)}
    \pgfmathsetmacro{\qtdMetadeMaisUm}{int(\qtd/2 + 1)}
    \pgfmathsetmacro{\qtdMetadeMaisTres}{int(\qtd/2 + 3)}
    \pgfmathsetmacro{\qtdMetadeMaisQuatro}{int(\qtd/2 + 4)}

    \pgfmathsetmacro{\qtdMenosTres}{int(\qtd - 3)}
    \pgfmathsetmacro{\qtdMenosDois}{int(\qtd - 2)}
    \pgfmathsetmacro{\qtdMenosUm}{int(\qtd-1)}

    \pgfmathsetmacro{\ang}{\qtdMetadeMaisUm*\espaco+\espacoMetade+\rot}
    \draw[very thick,dashed,black] (\ang:\raio + 1) -- (-90:\raio + 1);

    \foreach \x in {1,...,\qtd} {
        \pgfmathsetmacro{\ang}{\x*\espaco+\rot}
        \node (\x) [black vertex] at (\ang:\raio) {};
    }

    \foreach \x in {2,...,\qtd} {
        \pgfmathsetmacro {\y}{int(\qtd/2 + 2)}
        \ifthenelse{ \equal{\x}{\y} } {
            \draw[blue edge] (1) -- (\x);
        }{
            \draw[red edge] (1) -- (\x);
        }
    }

    \foreach \x in {2,...,\qtdMetadeMenosDois} {
        \pgfmathsetmacro{\xMaisUm}{int(\x + 1)}
        \pgfmathsetmacro{\fim}{int(\qtd - \x)}
        \foreach \y in {\xMaisUm,...,\fim} {
            \draw[red edge] (\x) -- (\y);
        }
    }

    \draw[blue edge] (\qtdMetadeMenosUm) -- (\qtdMetade);
    \draw[red edge] (\qtdMetadeMenosUm) -- (\qtdMetadeMaisUm);
    \draw[red edge] (\qtdMetadeMenosUm) -- (\qtdMetadeMaisDois);
    \draw[red edge] (\qtdMetade) -- (\qtdMetadeMaisDois);
    \draw[strong red edge] (\qtdMetadeMenosUm) -- (\qtdMenosUm);

    \draw[red edge] (\qtdMetade) -- (\qtdMetadeMaisUm);

    \pgfmathsetmacro{\qtdMetadeMaisTresMaisStep}{int(\qtdMetadeMaisTres + 2)}
    \ifthenelse{\qtdMetadeMaisTresMaisStep < \qtdMenosTres \OR \qtdMetadeMaisTresMaisStep = \qtdMenosTres}{
        \foreach \x in {\qtdMetadeMaisTres,\qtdMetadeMaisTresMaisStep,...,\qtdMenosTres} {
            \pgfmathsetmacro{\xMaisUm}{int(\x + 1)}
            \draw[strong red edge] (\x) -- (\xMaisUm);
        }
    }{
        \pgfmathsetmacro{\x}{int(\qtdMetadeMaisTres)}
        \pgfmathsetmacro{\xMaisUm}{int(\x + 1)}
        \draw[strong red edge] (\x) -- (\xMaisUm);
    }

    \pgfmathsetmacro{\blueDegree}{int(\qtdMetade)}
    \pgfmathsetmacro{\redDegree}{int(\qtd - 1 - \blueDegree)}
    \pgfmathsetmacro{\blueDegreeTwo}{int(\qtdMetade + 1)}
    \pgfmathsetmacro{\redDegreeTwo}{int(\qtd - 1 - \blueDegree - 1)}
  
    \node (1_label)              [ ] at (1 * \espaco+\rot:\raio + 1)                  {$v_{3}$ ({\color{red} $\qtdMenosDois$}, {\color{blue} $1$})};
    \node (metade_label)         [ ] at (\qtdMetade * \espaco+\rot:\raio + 1)         {$v_{1}$ ({\color{red} $\redDegree$}, {\color{blue} $\blueDegree$})};
    \node (metadeMaisUm_label)   [ ] at (\qtdMetadeMaisUm * \espaco+\rot:\raio + 1)   {$v_{\qtdMetadeMaisDois}$ ({\color{red} $\redDegree$}, {\color{blue} $\blueDegree$})};
    \node (metadeMaisDois_label)   [ ] at (\qtdMetadeMaisDois * \espaco+\rot:\raio + 1)   {$v_{\qtdMetadeMaisTres}$ ({\color{red} $\redDegreeTwo$}, {\color{blue} $\blueDegreeTwo$})};
    \node (metadeMaisTres_label) [ ] at (\qtdMetadeMaisTres * \espaco+\rot:\raio + 1) {$v_{2}$ ({\color{red} $\redDegreeTwo$}, {\color{blue} $\blueDegreeTwo$})};

    \foreach \x in {2,...,\qtdMetadeMenosUm} {
        \pgfmathsetmacro{\xMaisDois}{int(\x + 2)}
        \pgfmathsetmacro{\redDegree}{int(\qtd - \x - 1)}
        \pgfmathsetmacro{\blueDegree}{int(\x)}
        \node (\x_label) [ ] at (\x * \espaco+\rot:\raio + 1) {$v_{\xMaisDois}$ ({\color{red} $\redDegree$}, {\color{blue} $\blueDegree$})};
    }

    \foreach \x in {\qtdMetadeMaisQuatro,...,\qtd} {
        \pgfmathsetmacro{\redDegree}{int(\qtd - \x + 1)}
        \pgfmathsetmacro{\blueDegree}{int(\qtd - 1 - \redDegree)}
        \node (\x_label) [ ] at (\x * \espaco+\rot:\raio + 1) {$v_{\x}$ ({\color{red} $\redDegree$}, {\color{blue} $\blueDegree$})};
    }

    \node () [redconf vertex] at (\qtdMetade * \espaco+\rot:\raio) {};

    \node () [redconf vertex] at (\qtdMetadeMaisUm * \espaco+\rot:\raio) {};

    \node () [blueconf vertex] at (\qtdMetadeMaisDois * \espaco+\rot:\raio) {};

    \node () [blueconf vertex] at (\qtdMetadeMaisTres * \espaco+\rot:\raio) {};
  
\end{tikzpicture}
 \\
\begin{tikzpicture}[scale=0.8]

    \pgfmathsetmacro{\qtd}{int(12)}
    \pgfmathsetmacro{\raio}{2.7}

    \tikzset{black vertex/.style={circle,draw,minimum size=2mm,inner sep=0pt,outer sep=1pt,fill=black, color=black}}
    \tikzset{redconf vertex/.style={circle,minimum size=3mm,inner sep=0pt,outer sep=1pt,fill=red,draw=white}}
    \tikzset{blueconf vertex/.style={circle,minimum size=3mm,inner sep=0pt,outer sep=1pt,fill=blue,draw=white}}
    \tikzset{redblueconf vertex/.style={circle,minimum size=1mm,inner sep=0pt,outer sep=1pt,fill=blue,draw=blue}}

    \tikzstyle{blue edge}=[line width=4, blue]
    \tikzstyle{red edge}=[line width=1, red]
    \tikzstyle{strong red edge}=[line width=4, red]

    \pgfmathsetmacro{\espaco}{360/\qtd}
    \pgfmathsetmacro{\espacoMetade}{360/(\qtd*2)}    
    \pgfmathsetmacro{\rot}{-\espaco - 90 + \espacoMetade}
        
    \pgfmathsetmacro{\qtdMetade}{int(\qtd/2)}
    \pgfmathsetmacro{\qtdMetadeMenosDois}{int(\qtd/2 - 2)}
    \pgfmathsetmacro{\qtdMetadeMenosUm}{int(\qtd/2 -1)}
    \pgfmathsetmacro{\qtdMetadeMaisDois}{int(\qtd/2 + 2)}
    \pgfmathsetmacro{\qtdMetadeMaisUm}{int(\qtd/2 + 1)}
    \pgfmathsetmacro{\qtdMetadeMaisTres}{int(\qtd/2 + 3)}
    \pgfmathsetmacro{\qtdMetadeMaisQuatro}{int(\qtd/2 + 4)}

    \pgfmathsetmacro{\qtdMenosTres}{int(\qtd - 3)}
    \pgfmathsetmacro{\qtdMenosDois}{int(\qtd - 2)}
    \pgfmathsetmacro{\qtdMenosUm}{int(\qtd-1)}

    \pgfmathsetmacro{\ang}{\qtdMetadeMaisUm*\espaco+\espacoMetade+\rot}
    \draw[very thick,dashed,black] (\ang:\raio + 1) -- (-90:\raio + 1);

    \foreach \x in {1,...,\qtd} {
        \pgfmathsetmacro{\ang}{\x*\espaco+\rot}
        \node (\x) [black vertex] at (\ang:\raio) {};
    }

    \foreach \x in {2,...,\qtd} {
        \pgfmathsetmacro {\y}{int(\qtd/2 + 1)}
        \ifthenelse{ \equal{\x}{\y} } {
            \draw[blue edge] (1) -- (\x);
        }{
            \draw[red edge] (1) -- (\x);
        }
    }

    \foreach \x in {2,...,\qtdMetadeMenosDois} {
        \pgfmathsetmacro{\xMaisUm}{int(\x + 1)}
        \pgfmathsetmacro{\fim}{int(\qtd - \x)}
        \foreach \y in {\xMaisUm,...,\fim} {
            \draw[red edge] (\x) -- (\y);
        }
    }

    \draw[blue edge] (\qtdMetadeMenosUm) -- (\qtdMetade);
    \draw[red edge] (\qtdMetadeMenosUm) -- (\qtdMetadeMaisUm);
    \draw[strong red edge] (\qtdMetadeMenosUm) -- (\qtdMenosUm);

    \draw[red edge] (\qtdMetade) -- (\qtdMetadeMaisUm);

    \pgfmathsetmacro{\qtdMetadeMaisTresMaisStep}{int(\qtdMetadeMaisTres + 2)}
    \ifthenelse{\qtdMetadeMaisTresMaisStep < \qtdMenosTres \OR \qtdMetadeMaisTresMaisStep = \qtdMenosTres}{
        \foreach \x in {\qtdMetadeMaisTres,\qtdMetadeMaisTresMaisStep,...,\qtdMenosTres} {
            \pgfmathsetmacro{\xMaisUm}{int(\x + 1)}
            \draw[strong red edge] (\x) -- (\xMaisUm);
        }
    }{
        \pgfmathsetmacro{\x}{int(\qtdMetadeMaisTres)}
        \pgfmathsetmacro{\xMaisUm}{int(\x + 1)}
        \draw[strong red edge] (\x) -- (\xMaisUm);
    }

    \pgfmathsetmacro{\blueDegree}{int(\qtdMetade)}
    \pgfmathsetmacro{\redDegree}{int(\qtd - 1 - \blueDegree)}
    \pgfmathsetmacro{\blueDegreeTwo}{int(\qtdMetade + 1)}
    \pgfmathsetmacro{\redDegreeTwo}{int(\qtd - 1 - \blueDegree - 1)}
  
    \node (1_label)              [ ] at (1 * \espaco+\rot:\raio + 1)                  {$v_{3}$ ({\color{red} $\qtdMenosDois$}, {\color{blue} $1$})};
    \node (metade_label)         [ ] at (\qtdMetade * \espaco+\rot:\raio + 1)         {$v_{1}$ ({\color{red} $\redDegree$}, {\color{blue} $\blueDegree$})};
    \node (metadeMaisUm_label)   [ ] at (\qtdMetadeMaisUm * \espaco+\rot:\raio + 1)   {$v_{\qtdMetadeMaisDois}$ ({\color{red} $\redDegree$}, {\color{blue} $\blueDegree$})};
    \node (metadeMaisDois_label) [ ] at (\qtdMetadeMaisDois * \espaco+\rot:\raio + 1) {$v_{\qtdMetadeMaisTres}$ ({\color{red} $\redDegreeTwo$}, {\color{blue} $\blueDegreeTwo$})};
    \node (metadeMaisTres_label) [ ] at (\qtdMetadeMaisTres * \espaco+\rot:\raio + 1) {$v_{2}$ ({\color{red} $\redDegreeTwo$}, {\color{blue} $\blueDegreeTwo$})};

    \foreach \x in {2,...,\qtdMetadeMenosUm} {
        \pgfmathsetmacro{\xMaisDois}{int(\x + 2)}
        \pgfmathsetmacro{\redDegree}{int(\qtd - \x - 1)}
        \pgfmathsetmacro{\blueDegree}{int(\x)}
        \node (\x_label) [ ] at (\x * \espaco+\rot:\raio + 1) {$v_{\xMaisDois}$ ({\color{red} $\redDegree$}, {\color{blue} $\blueDegree$})};
    }

    \foreach \x in {\qtdMetadeMaisQuatro,...,\qtd} {
        \pgfmathsetmacro{\redDegree}{int(\qtd - \x + 1)}
        \pgfmathsetmacro{\blueDegree}{int(\qtd - 1 - \redDegree)}
        \node (\x_label) [ ] at (\x * \espaco+\rot:\raio + 1) {$v_{\x}$ ({\color{red} $\redDegree$}, {\color{blue} $\blueDegree$})};
    }

    \node () [redconf vertex] at (\qtdMetade * \espaco+\rot:\raio) {};

    \node () [redconf vertex] at (\qtdMetadeMaisUm * \espaco+\rot:\raio) {};

    \node () [blueconf vertex] at (\qtdMetadeMaisDois * \espaco+\rot:\raio) {};

    \node () [blueconf vertex] at (\qtdMetadeMaisTres * \espaco+\rot:\raio) {};
  
\end{tikzpicture}
     \hfill
\begin{tikzpicture}[scale=0.8]

    \pgfmathsetmacro{\qtd}{int(13)}
    \pgfmathsetmacro{\raio}{2.7}

    \tikzset{black vertex/.style={circle,draw,minimum size=2mm,inner sep=0pt,outer sep=1pt,fill=black, color=black}}
    \tikzset{redconf vertex/.style={circle,minimum size=3mm,inner sep=0pt,outer sep=1pt,fill=red,draw=white}}
    \tikzset{blueconf vertex/.style={circle,minimum size=3mm,inner sep=0pt,outer sep=1pt,fill=blue,draw=white}}
    \tikzset{redblueconf vertex/.style={circle,minimum size=1mm,inner sep=0pt,outer sep=1pt,fill=blue,draw=blue}}

    \tikzstyle{blue edge}=[line width=4, blue]
    \tikzstyle{red edge}=[line width=1, red]
    \tikzstyle{strong red edge}=[line width=4, red]

    \pgfmathsetmacro{\espaco}{360/\qtd}
    \pgfmathsetmacro{\espacoMetade}{360/(\qtd*2)}    
    \pgfmathsetmacro{\rot}{-\espaco - 90 + \espacoMetade}
        
    \pgfmathsetmacro{\qtdMetade}{int(\qtd/2)}
    \pgfmathsetmacro{\qtdMetadeMenosDois}{int(\qtd/2 - 2)}
    \pgfmathsetmacro{\qtdMetadeMenosUm}{int(\qtd/2 -1)}
    \pgfmathsetmacro{\qtdMetadeMaisDois}{int(\qtd/2 + 2)}
    \pgfmathsetmacro{\qtdMetadeMaisUm}{int(\qtd/2 + 1)}
    \pgfmathsetmacro{\qtdMetadeMaisTres}{int(\qtd/2 + 3)}
    \pgfmathsetmacro{\qtdMetadeMaisQuatro}{int(\qtd/2 + 4)}
    \pgfmathsetmacro{\qtdMetadeMaisCinco}{int(\qtd/2 + 5)}

    \pgfmathsetmacro{\qtdMenosCinco}{int(\qtd - 5)}
    \pgfmathsetmacro{\qtdMenosQuatro}{int(\qtd - 4)}
    \pgfmathsetmacro{\qtdMenosTres}{int(\qtd - 3)}
    \pgfmathsetmacro{\qtdMenosDois}{int(\qtd - 2)}
    \pgfmathsetmacro{\qtdMenosUm}{int(\qtd-1)}

    \pgfmathsetmacro{\ang}{\qtdMetade*\espaco+\espacoMetade+\rot}
    \draw[very thick,dashed,black] (\ang:\raio + 1) -- (-90:\raio + 1);

    \foreach \x in {1,...,\qtd} {
        \pgfmathsetmacro{\ang}{\x*\espaco+\rot}
        \node (\x) [black vertex] at (\ang:\raio) {};
    }

    \foreach \x in {2,...,\qtd} {
        \pgfmathsetmacro {\y}{int(\qtd/2 + 1)}
        \ifthenelse{ \equal{\x}{\y} } {
            \draw[blue edge] (1) -- (\x);
        }{
            \draw[red edge] (1) -- (\x);
        }
    }

    \foreach \x in {2,...,\qtdMetadeMenosDois} {
        \pgfmathsetmacro{\xMaisUm}{int(\x + 1)}
        \pgfmathsetmacro{\fim}{int(\qtd - \x)}
        \foreach \y in {\xMaisUm,...,\fim} {
            \draw[red edge] (\x) -- (\y);
        }
    }

    \draw[blue edge] (\qtdMetadeMenosUm) -- (\qtdMetade);
    \draw[red edge] (\qtdMetadeMenosUm) -- (\qtdMetadeMaisUm);
    \draw[red edge] (\qtdMetadeMenosUm) -- (\qtdMetadeMaisDois);
    \draw[strong red edge] (\qtdMetadeMenosUm) -- (\qtdMenosUm);

    \draw[red edge] (\qtdMetade) -- (\qtdMetadeMaisUm);

    \pgfmathsetmacro{\qtdMetadeMaisDoisMaisStep}{int(\qtdMetadeMaisDois + 2)}
    \ifthenelse{\qtdMetadeMaisDoisMaisStep < \qtdMenosCinco \OR \qtdMetadeMaisDoisMaisStep = \qtdMenosCinco}{
        \foreach \x in {\qtdMetadeMaisDois,\qtdMetadeMaisDoisMaisStep,...,\qtdMenosCinco} {
            \pgfmathsetmacro{\xMaisUm}{int(\x + 1)}
            \draw[strong red edge] (\x) -- (\xMaisUm);
        }
    }{
        \pgfmathsetmacro{\x}{int(\qtdMetadeMaisDois)}
        \pgfmathsetmacro{\xMaisUm}{int(\x + 1)}
        \draw[strong red edge] (\x) -- (\xMaisUm);
    }

    \draw[strong red edge] (\qtdMenosUm) -- (\qtdMetadeMenosUm);
    \draw[strong red edge] (\qtdMenosDois) -- (\qtdMetade);
    \draw[strong red edge] (\qtdMenosTres) -- (\qtdMetadeMaisUm);

    \pgfmathsetmacro{\blueDegree}{int(\qtdMetade)}
    \pgfmathsetmacro{\redDegree}{int(\qtd - 1 - \blueDegree)}
  
    \node (1_label)              [ ] at (1 * \espaco+\rot:\raio + 1)                  {$v_{3}$ ({\color{red} $\qtdMenosDois$}, {\color{blue} $1$})};
    \node (metade_label)         [ ] at (\qtdMetade * \espaco+\rot:\raio + 1)         {$v_{1}$ ({\color{red} $\redDegree$}, {\color{blue} $\blueDegree$})};
    \node (metadeMaisUm_label)   [ ] at (\qtdMetadeMaisUm * \espaco+\rot:\raio + 1)   {$v_{\qtdMetadeMaisDois}$ ({\color{red} $\redDegree$}, {\color{blue} $\blueDegree$})};
    \node (metadeMaisDois_label) [ ] at (\qtdMetadeMaisDois * \espaco+\rot:\raio + 1) {$v_{2}$ ({\color{red} $\redDegree$}, {\color{blue} $\blueDegree$})};

    \foreach \x in {2,...,\qtdMetadeMenosUm} {
        \pgfmathsetmacro{\xMaisDois}{int(\x + 2)}
        \pgfmathsetmacro{\redDegree}{int(\qtd - \x - 1)}
        \pgfmathsetmacro{\blueDegree}{int(\x)}
        \node (\x_label) [ ] at (\x * \espaco+\rot:\raio + 1) {$v_{\xMaisDois}$ ({\color{red} $\redDegree$}, {\color{blue} $\blueDegree$})};
    }

    \foreach \x in {\qtdMetadeMaisTres,...,\qtd} {
        \pgfmathsetmacro{\redDegree}{int(\qtd - \x + 1)}
        \pgfmathsetmacro{\blueDegree}{int(\qtd - 1 - \redDegree)}
        \node (\x_label) [ ] at (\x * \espaco+\rot:\raio + 1) {$v_{\x}$ ({\color{red} $\redDegree$}, {\color{blue} $\blueDegree$})};
    }

    \node () [redconf vertex] at (\qtdMetade * \espaco+\rot:\raio) {};
    \node () [redblueconf vertex] at (\qtdMetade * \espaco+\rot:\raio) {};

    \node () [redconf vertex] at (\qtdMetadeMaisUm * \espaco+\rot:\raio) {};

    \node () [blueconf vertex] at (\qtdMetadeMaisDois * \espaco+\rot:\raio) {};
  
\end{tikzpicture}
 
    \caption{Strange colorings of $E(K_n)$ with $n \in \{10, 11, 12, 13\}$, where $V(K_n) = \{v_1, \ldots, v_n\}$ and the sequence is $(v_3,$ $\ldots,$ $v_{\floor{n/2}+1},$ $v_1,$ $v_{\floor{n/2}+2},$ $v_{\floor{n/2}+3},$ $v_2,$ $v_{\floor{n/2}+4},$ $\ldots,$ $v_n)$ if $\ceil{n/2}$ is even, and it is $(v_3,$ $\ldots,$ $v_{\floor{n/2}+1},$ $v_1,$ $v_{\floor{n/2}+2},$ $v_2,$ $v_{\floor{n/2}+3},$ $\ldots,$ $v_n)$ otherwise.
    Strange edges are highlighted, as well as the vertices of the conflicting edges.
    For better visualization, we omit the other blue edges.}
    \label{fig:strange}
\end{figure}
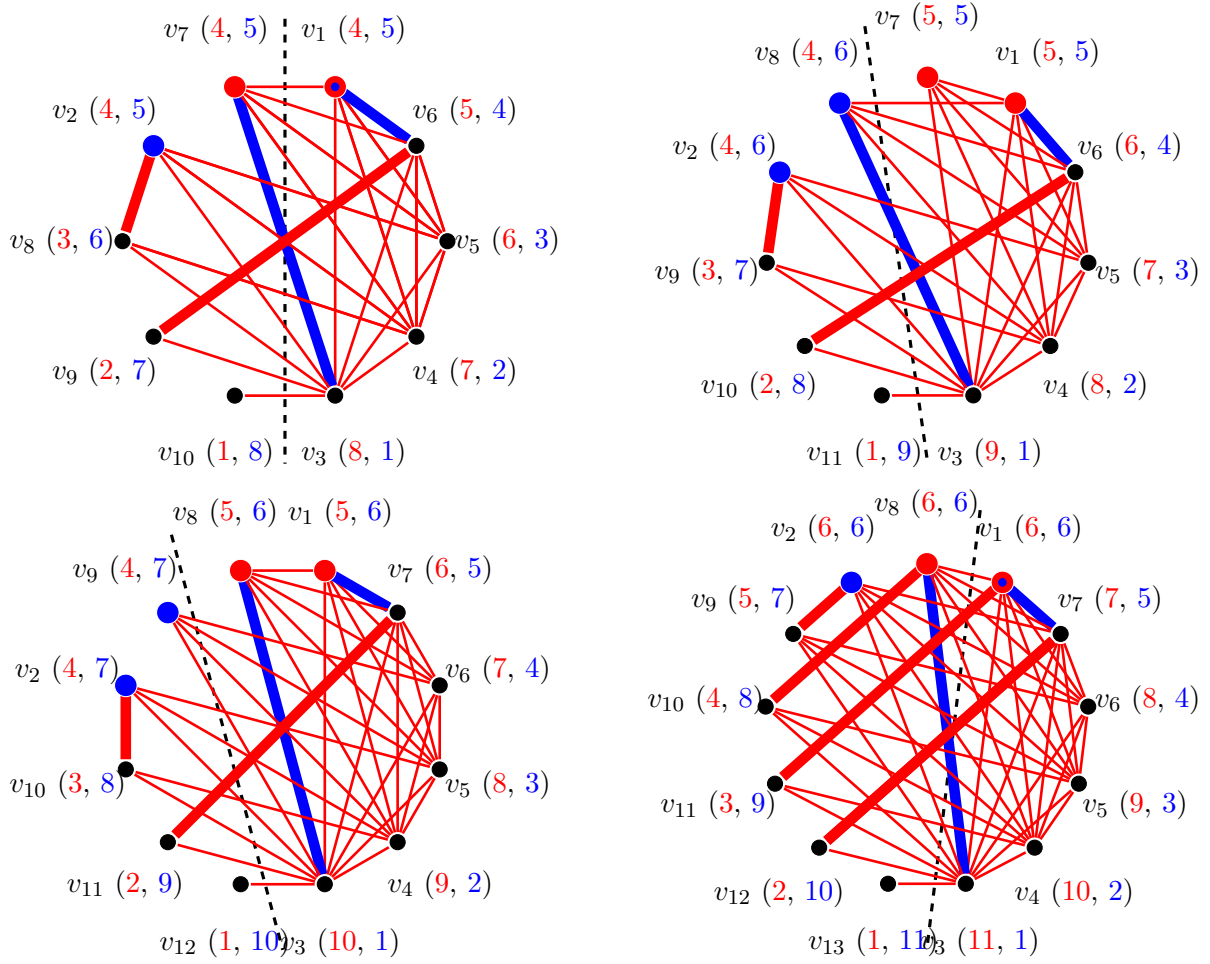%

Given a graph $G$ and a $2$-edge coloring $\phi \colon E(G) \to \{\red, \blue\}$, we say an even cycle $C = (u_1, \ldots, u_k, u_{k+1} = u_1)$ is an \emph{alternating cycle in $\phi$} if $\phi(u_iu_{i+1}) \neq \phi(u_{i+1}u_{i+2})$, for any $1 \leq i < k$, and $d_{G_\gamma}(u_i) \neq d_{G_\gamma}(u_{i + 1})$, for any $1 \leq i \leq k$ and $\gamma \in \{\red, \blue\} \setminus \phi(u_iu_{i + 1})$.
In other words, an alternating cycle has incident edges with different colors and the endpoints of a red edge (resp.\ blue edge) have different blue degrees (resp.\ red degrees).
Let $\phi'$ be the $2$-edge coloring of $E(G)$ such that $\phi'(uv) \in \{\red, \blue\} \setminus \phi(uv)$ if $uv \in E(C)$ and $\phi'(uv) = \phi(uv)$ otherwise.
We say $\phi'$ is \emph{obtained from $\phi$ by inverting $C$}.

Lemma~\ref{lem:alternating-cycle} shows that an inversion on an alternating cycle does not create conflicting edges.
It will be used in the proof of Lemma~\ref{lemma:splitlarge-3}.

\begin{lemma}%
\label{lem:alternating-cycle}
    Let $G$ be a graph, $\phi \colon E(G) \to \{\red, \blue\}$ be a $2$-edge coloring of $G$.
    Let $C$ be an alternating cycle in $\phi$ and let $\phi'$ be obtained from $\phi$ by inverting $C$.
    The set of conflicting edges of $G_{\gamma,\phi'}$ is a subset of the conflicting edges of $G_{\gamma,\phi}$, for any $\gamma \in \{\red, \blue\}$.
\end{lemma}
\begin{proof}
    First note that, in $\phi$, every vertex of $V(C)$ has one incident blue edge and one incident red edge, which remains valid in $\phi'$.
    No other edge has its color changed.
    Thus, $d_{G_{\gamma,\phi}}(v) = d_{G_{\gamma,\phi'}}(v)$ for any $v \in V(G)$ and $\gamma \in \{\red,\blue\}$.

    Let $\gamma \in \{\red,\blue\}$, $\bar{\gamma} \in \{\red,\blue\} \setminus \{\gamma\}$, and consider any edge $uv \in E(G)$.
    If $uv \notin E(C)$, then $uv$ is conflicting in $G_{\gamma,\phi}$ (resp.\ $G_{\bar{\gamma},\phi}$) if and only if it is conflicting in $G_{\gamma,\phi'}$ (resp.\ $G_{\bar{\gamma},\phi'}$).
    So let $uv \in E(C)$ and let $\phi(uv) = \gamma$.
    From the definition of alternating cycle, $d_{G_{\bar{\gamma},\phi}}(u) \neq d_{G_{\bar{\gamma},\phi}}(v)$.
    This means that $uv$ is not conflicting in $G_{\bar{\gamma},\phi}$, no matter if it is conflicting in $G_{\gamma,\phi}$ or not.
\end{proof}%

\begin{lemma}%
\label{lemma:splitlarge-3}
    Let $n\geq 10$ and let $G(X,Y)$ be a split graph with $X = \{v_1, \ldots, v_n\}$ where $d_1 \geq \cdots \geq d_n$ and $d_{\floor{n/2}} = 0$.
    If $d_1 \geq \floor{n/2}$ or $d_2 \geq 1$, then $\chii(G) = 2$.
\end{lemma}
\begin{proof}
    First note that since $d_{\floor{n/2}} = 0$ we have $\chii(G) \geq 2$ due to Fact~\ref{fact:splitlarge-0}.

    Let $\phi' \colon E(G[X]) \to \{\red,\blue\}$ be the strange coloring of $E(G[X])$ for the sequence
    \begin{equation}
        \pi =
        \begin{cases}
            (v_3, \ldots, v_{\floor{n/2}+1}, v_1, v_{\floor{n/2}+2}, v_{\floor{n/2}+3}, v_2, v_{\floor{n/2}+4}, \ldots, v_n) & \text{if } \ceil{n/2} \text{ is even}\\
            (v_3, \ldots, v_{\floor{n/2}+1}, v_1, v_{\floor{n/2}+2},                    v_2, v_{\floor{n/2}+3}, \ldots, v_n) & \text{otherwise} \enspace .
        \end{cases}
    \end{equation}

    We will show how to obtain a locally irregular $2$-edge coloring of $G$ from $\phi'$.
    We start by extending $\phi'$ to a coloring $\phi$ of $E(G)$ with colors red and blue.
    Let $X_\red = \{v_1, v_3, v_4, \ldots, v_{\floor{n/2} - 1}\}$.
    We give color red to all edges between $X_\red$ and $Y$, and color blue to all edges between~$v_2$ and~$Y$.
    Note that this is a coloring of $E(G)$ because $d_{\floor{n/2}} = 0$.

    We start by showing that, in both $G_{\red, \phi}$ and $G_{\blue, \phi}$, there is no conflicting edge $xy$ for $x \in X$ and $y \in Y$.
    By the construction of $\phi$ and since $d_{\floor{n/2}} = 0$, the neighborhood of any vertex $y \in Y$ in $G_{\red, \phi}$ is contained in $X_\red$, and hence $d_{G_{\red, \phi}}(y) \leq \floor{n/2} - 2$.
    Since $d_1 \geq 1$, we have $d_{G_{\red, \phi}}(x) \geq \floor{n/2}$ for all $x \in X_\red$.
    Therefore, $d_{G_{\red, \phi}}(x) > d_{G_{\red, \phi}}(y)$ for any $x \in X$ and $y \in Y$.
    Now note that, also by the construction of $\phi$, the only vertex in $X$ which can have a neighbor in $Y$ is $v_2$, which implies that $d_{G_{\blue, \phi}}(y) \leq 1$ for all $y \in Y$.
    Since $d_2 \geq 1$, we have $d_{G_{\blue, \phi}}(v_2) \geq \floor{n/2} + 1$, so the graph $G_{\blue, \phi}$ has no conflicting edge $xy$ with $x \in X$ and $y \in Y$.
    Therefore, if $\phi$ is not a locally irregular edge coloring, it is because there is a conflicting edge between a pair of vertices in $X$.

    Remark that the only difference between the sequence $\pi$ used to build $\phi$ is the amount of vertices between $v_1$ and $v_2$ in it.
    To avoid being repetitive, in the remainder of this proof we will consider that $\ceil{n/2}$ is odd.
    The case where $n$ is even is analogous.

    Now we show that there exists a $2$-edge coloring $\zeta$ such that $G_{\red, \zeta}$ is locally irregular.
    If $G_{\red, \phi}$ is locally irregular, then let $\zeta = \phi$.
    Thus, suppose that $G_{\red, \phi}$ has at least one conflicting edge $uv$.
    Since 
    \begin{align*}
        d_{G_{\red, \phi'}}(v_3) & > d_{G_{\red, \phi'}}(v_4) > \cdots >  d_{G_{\red, \phi'}}(v_{\floor{n/2} + 1}) > d_{G_{\red, \phi'}}(v_1) \\
                                 & = d_{G_{\red, \phi'}}(v_{\floor{n/2} + 2}) =  d_{G_{\red, \phi'}}(v_2) 
                                 > d_{G_{\red, \phi'}}(v_{\floor{n/2} + 3}) > \cdots > d_{G_{\red, \phi'}}(v_n) \enspace,
    \end{align*}
    
    \begin{equation*}
        d_1 \geq d_2 \geq d_3 \geq \cdots \geq 1 > d_{\floor{n/2}} = \cdots = d_n = 0 \enspace ,
    \end{equation*}
    
    \begin{equation*}
        d_1 \geq d_2 \geq 1 \enspace ,
    \end{equation*}
    and
    \begin{equation*}
        d_{G_{\red, \phi}} (v_i) =  d_{G_{\red, \phi'}} (v_i) + d_i \text{ for } i \neq 2 \quad \tand \quad d_{G_{\red, \phi}} (v_2) =  d_{G_{\red, \phi'}} (v_2) \enspace ,
    \end{equation*}
    we have that
    \begin{equation}
    \label{enu:eq1}
        \begin{array}{rl}
            d_{G_{\red, \phi}}(v_3) & > d_{G_{\red, \phi}}(v_4) > \cdots >  d_{G_{\red, \phi}}(v_{\floor{n/2} + 1}) \\
                                    & > d_{G_{\red, \phi}}(v_{\floor{n/2} + 2}) = d_{G_{\red, \phi}}(v_2)
                                    >  d_{G_{\red, \phi}}(v_{\floor{n/2} + 3}) > \cdots > d_{G_{\red, \phi}}(v_{n})
        \end{array}
    \end{equation}
    and 
    \begin{equation}
        \label{enu:eq2}
        d_{G_{\red, \phi}}(v_1) > d_{G_{\red, \phi}}(v_{\floor{n/2} + 2}) \enspace .
    \end{equation}

    By~\eqref{enu:eq1} and since $\phi(v_{\floor{n/2} + 2}v_2)$ is blue, we may assume that $u = v_1$, and by~\eqref{enu:eq1} and~\eqref{enu:eq2}, we have $v \in \{v_3, \ldots, v_{\floor{n/2} + 1}\}$.
    As a result, $v_1v$ is the only conflicting edge of $G_{\red, \phi}$.
    Moreover, note that the edge $v_1v_{\floor{n/2} + 1}$ cannot be a conflicting edge, since $\phi(v_1v_{\floor{n/2} + 1})$ is blue.
    Therefore, $v \in \{v_3, \ldots, v_{\floor{n/2}}\}$, and thus $d_1 \geq 2$.

    We will show that there exists an alternating cycle $C$ such that $G_{\red,\zeta}$ is locally irregular, where $\zeta$ is the coloring obtained from $\phi$ by inverting $C$.
    We start remarking that
    \begin{itemize}
        \item $\phi(v_1v)$ is red, $d_{G_{\blue,\phi}}(v_1) = \floor{n/2}$, and $d_{G_{\blue,\phi}}(v) \leq \floor{n/2} - 2$,
        \item $\phi(v_{\floor{n/2}+ 1}v_1)$ is blue, $d_{G_{\red, \phi}}(v_{\floor{n/2}+1}) = \ceil{n/2} + d_{\floor{n/2}+1} = \ceil{n/2} + 0$ (because $d_{\floor{n/2}} = 0$ and $d_1 \geq \dots \geq d_n$), and $d_{G_{\red, \phi}}(v_1) = \ceil{n/2} - 1 + d_1 \geq \ceil{n/2} + 1$, since $d_1 \geq 2$.
    \end{itemize}

    We claim that if $v \neq v_3$, then $C_1 = (v_1, v, v_{n - 1}, v_{\floor{n/2} + 1}, v_1)$ is an alternating cycle in $\phi$, otherwise $C_2 = (v_1, v_3, v_k, v_{\floor{n/2} + 1}, v_1)$, where $v_k$ is the only neighbor of $v_3$ in $G_{\blue, \phi}$, is an alternating cycle in $\phi$.
    First, suppose that $v \neq v_3$, and note that $\phi(vv_{n - 1})$ is blue for any $v \in \{v_3, \ldots, v_{\floor{n/2}}\}$, $d_{G_{\red, \phi}}(v) \geq \floor{n/2} + 1$, and $d_{G_{\red, \phi}}(v_{n - 1}) = 2$.
    In addition, note that $\phi(v_{n-1}v_{\floor{n/2} + 1})$ is red, $d_{G_{\blue,\phi}}(v_{n-1}) = n - 3$, $d_{G_{\blue,\phi}}(v_{\floor{n/2} + 1}) = \floor{n/2} - 1$.
    Therefore, $C = C_1$ is the desired alternating cycle if $v \neq v_3$.
    
    Suppose that $v = v_3$, and note that $d_{G_{\red, \phi}}(v_3) \geq n - 2$ and $d_{G_{\red, \phi}}(v_k) \leq \floor{n/2}$.
    Moreover, note that $\phi(v_kv_{\floor{n/2} + 1})$ is red, $d_{G_{\blue,\phi}}(v_k) = \floor{n/2}$, and $d_{G_{\blue,\phi}}(v_{\floor{n/2} + 1}) = \floor{n/2} - 1$.
    Hence, $C = C_2$ is the desired alternating cycle.
    Thus, let $\zeta$ be the coloring obtained from $\phi$ by inverting $C$.
    By Lemma~\ref{lem:alternating-cycle}, we have $d_{G_{\gamma, \zeta}}(v) = d_{G_{\gamma, \phi}}(v)$ for every vertex $v \in V(G)$ and $\gamma \in \{\red, \blue\}$ and no new conflicting edge in $\red$ was created.
    Since $\zeta(v_1v)$ is blue, $G_{\red, \zeta}$ is locally irregular.

    Now we show how to build a locally irregular $2$-edge coloring $\theta$ from $\zeta$.
    Since $G_{\red, \zeta}$ is locally irregular, if $G_{\blue, \zeta}$ is also locally irregular, then let $\theta = \zeta$ and the result follows.
    Thus, $G_{\blue, \zeta}$ has at least one conflicting edge $uv$.

    Note that
    \begin{equation*}
        d_{G_{\blue,\zeta}}(v) = d_{G_{\blue,\phi}}(v) = d_{G_{\blue,\phi'}}(v) \text{ for any } v \neq v_2 
    \end{equation*}
    and
    \begin{equation*}
        d_{G_{\blue,\zeta}}(v_2) = d_{G_{\blue,\phi}}(v_2) = d_{G_{\blue,\phi'}}(v_2) + d_2 \enspace .
    \end{equation*}
    Therefore,
    \begin{equation}
    \label{enu2:eq1}
        \begin{array}{rl}
            d_{G_{\blue, \zeta}}(v_n) & > d_{G_{\blue, \zeta}}(v_{n-1}) > \cdots >  d_{G_{\blue, \zeta}}(v_{\floor{n/2} + 3}) \\
                                     & > d_{G_{\blue, \zeta}}(v_{\floor{n/2} + 2}) = d_{G_{\blue, \zeta}}(v_1) 
                                      > d_{G_{\blue, \zeta}}(v_{\floor{n/2} + 1}) > \cdots > d_{G_{\blue, \zeta}}(v_3) \enspace.
        \end{array}
    \end{equation}

    Since $\zeta(v_{\floor{n/2}+2}v_1)$ is red, the conflicting edges in $G_{\blue,\zeta}$ must involve $v_2$.
    By~\eqref{enu2:eq1}, there can be only one conflicting edge in $G_{\blue,\zeta}$ between $v_2$ and $v \in \{v_{\floor{n/2}+3}, \ldots, v_n\}$.
    Moreover, note that the edge $v_2v_{\floor{n/2} + 3}$ cannot be a conflicting edge, since $\zeta(v_2v_{\floor{n/2} + 3})$ is red.
    Therefore, $v \in \{v_{\floor{n/2}+3}, \ldots, v_n\}$, and thus $d_2 \geq 2$.

    We will show that there exists an alternating cycle $C$ such that $G_{\blue,\theta}$ is locally irregular, where $\theta$ is the coloring obtained from $\zeta$ by inverting $C$.
    We start by remarking that $v_3$ has precisely one neighbor $v_k$ in $G_{\blue, \zeta}$ (it is either $v_1$ or $v_{\floor{n/2}+2}$) and that $\zeta(v_kv_{\floor{n/2} + 1})$ is red.
    Also,
    \begin{itemize}
        \item $d_{G_{\blue,\zeta}}(v_k) = \floor{n/2}$, and $d_{G_{\blue, \zeta}}(v_{\floor{n/2}+1}) = \floor{n/2} - 1$,
        \item $d_{G_{\red, \zeta}}(v_3) \geq n-2$ and either $v_k = v_1$, in which case the construction of $\zeta$ guarantees $d_{G_{\red,\zeta}}(v_1) \neq d_{G_{\red,\zeta}}(v_3)$, or $v_k = v_{\floor{n/2}+2}$, in which case $d_{G_{\blue,\zeta}}(v_{\floor{n/2}+2}) = \floor{n/2}$.
    \end{itemize}

    Note that $C = (v_2, v, v_3, v_k, v_{\floor{n/2}+1}, v_{\floor{n/2}+3}, v_2)$ is an alternating cycle in $\zeta$:
    \begin{itemize}
        \item $\zeta(v_2v) = \blue$, $d_{G_{\red,\zeta}}(v_2) = \floor{n/2}$, and $d_{G_{\red,\zeta}}(v) < \floor{n/2}$;
        \item $\zeta(vv_3) = \red$, $d_{G_{\blue,\zeta}}(v) > \floor{n/2}$, and $d_{G_{\blue,\zeta}}(v_3) = 1$;
        \item $\zeta(v_{\floor{n/2}+1}v_{\floor{n/2}+3}) = \blue$, $d_{G_{\red,\zeta}}(v_{\floor{n/2}+1}) \geq \floor{n/2}$, and $d_{G_{\red,\zeta}}(v_{\floor{n/2}+3}) \leq \floor{n/2}-1$; and
        \item $\zeta(v_{\floor{n/2}+3}v_2) = \red$, $d_{G_{\blue,\zeta}}(v_{\floor{n/2}+3}) \leq d_{G_{\blue,\zeta}}(v_2) + 1$.
    \end{itemize}

    Thus, let $\theta$ be the coloring obtained from $\zeta$ by inverting $C$.
    By Lemma~\ref{lem:alternating-cycle}, we have $d_{G_{\gamma, \theta}}(v) = d_{G_{\gamma, \zeta}}(v)$ for every vertex $v \in V(G)$ and $\gamma \in \{\red, \blue\}$ and no new conflicting edge in $\red$ or $\blue$ was created.
    Since $\theta(v_2v)$ is red, $G_{\blue, \theta}$ is locally irregular.
    It follows that $\theta$ is a locally irregular $2$-edge coloring of $G$.
\end{proof}%

\section{Decomposing split graphs with a small maximal clique}
\label{sec:small-split}

For completeness we also describe the characterization of the irregular chromatic index of split graphs that have a maximal clique with at most~9 vertices.

\begin{theorem}
\label{thm:small}
    Let $G(X,Y)$ be a split graph with $X = \{v_1, \ldots, v_n\}$ where $d_1 \geq \cdots \geq d_n$.
    If $n\leq 9$, then the following holds
\begin{enumerate}[label=\rmlabel]
    \item\label{it:main-2-1} $G$ is not decomposable if $G$ is the $K_2$, $K_3$ or, $P_4$ (the path with~$4$ vertices);
    \item If $G$ is decomposable, then 
    \begin{enumerate}
      \item\label{it:main-2-0} If $d_1 > d_2 > \cdots > d_n$, then $\chii(G) = 1$;
      \item\label{it:main-2-4} If $3 \leq n \leq 9$ and $\sum_{i=1}^{n} d_i \geq \floor{n/2}$, then $\chii(G) = 2$;
      \item\label{it:main-2-6} If $8 \leq n \leq 9$, $\sum_{i=1}^{3} d_i = 3$, and $d_2 \geq 1$, then $\chii(G) = 2$;
      \item\label{it:main-2-7} If $n = 9$ and $d_1 = d_2 = 1$, then $\chii(G) = 2$;
      \item\label{it:main-2-5} For all the other cases, it follows that $\chii(G) = 3$.
    \end{enumerate}
\end{enumerate}
\end{theorem}

\begin{proof}%
    If $|V(G)|~= 2$, then $G \simeq K_2$, and hence~\ref{it:main-2-1} holds.
    Thus, we may assume that $|V(G)|~\geq 3$.
    By Fact~\ref{fact:splitlarge-0}, we may assume that $n \geq 2$ and that there is an $i$ with $1 \leq i \leq n - 1$, such that $d_i = d_{i + 1}$, otherwise~\ref{it:main-2-0} holds.
    Moreover, from now on we know that if $G$ is decomposable, then $\chii(G) \geq 2$.

    First suppose that $n = 2$, and hence $d_1 = d_2$.
    If $d_1 = d_2 = 1$, then $G \simeq P_4$, since $X$ is a maximal clique, and, as a result,~\ref{it:main-2-1}~holds.
    Otherwise, $d_1 = d_2 \geq 2$, and hence $G$ is a bistar, and it is not hard to see that there exists a locally irregular 2-edge coloring for $G$.
    Thus, $\chii(G) = 2$ and~\ref{it:main-2-4} holds.
    Suppose $n = 3$.
    If $d_1 = 0$, then $G \simeq K_3$, and hence~\ref{it:main-2-1} holds.
    If $d_1 \geq 1$, then $\chii(G) \leq 2$ by Lemma~\ref{lemma:splitlarge-2}, and hence~\ref{it:main-2-4} follows.

    Let $n \in \{4, 5\}$.
    If $d_2 \geq 1$, then $\chii(G) = 2$ by Lemma~\ref{lemma:splitlarge-2}, and hence~\ref{it:main-2-4} follows.
    So we may assume that $d_2 = 0$.
    If $d_1 = 1$, then by Lemma~\ref{lemma:splitlarge-1}, we have $\chi(G) = 3$ and thus~\ref{it:main-2-5} follows.
    If $d_1 \geq 2$, then let $\phi \colon E(G[X]) \to \{\red, \blue\}$ be a normal coloring to the sequence $(v_2, v_1, v_4, v_3)$ if $n = 4$, or the sequence $(v_2, v_3, v_1, v_4, v_5)$ if $n = 5$.
    Let $\phi' \colon E(G) \to \{\red, \blue\}$ be the coloring obtained from $\phi$ by giving the color $\phi(v_1v_4)$ to all the edges in $E(v_1, Y)$.
    Note that the largest degree in $G_{\phi(v_1v_4), \phi}$ is $3$, and $d_{G_{\phi(v_1v_4),\phi'}}(v_1) \geq 4$.
    Thus, $\phi'$ is a locally irregular 2-edge coloring for $G$, i.e., $\chii(G) = 2$ and hence~\ref{it:main-2-4} follows.

    Let $n \in \{6, 7\}$.
    If $d_3 \geq 1$, then $\chii(G) = 2$ by Lemma~\ref{lemma:splitlarge-2}, and hence~\ref{it:main-2-4} follows.
    So we may assume that $d_3 = 0$.
    If $d_1 < 3$ and $d_2 = 0$, then $\chii(G) = 3$ by Lemma~\ref{lemma:splitlarge-1} and thus~\ref{it:main-2-5} follows.
    Thus, $d_1 \geq 3$ or $d_2 \geq 1$.
    If $d_1 \geq 3$, then let $\phi \colon E(G[X]) \to \{\red,\blue\}$ be a normal coloring to the sequence $(v_2, v_3, v_4, v_1, v_5, v_6)$ if $n = 6$, or the sequence $(v_2, v_3, v_4, v_1, v_5, v_6, v_7)$ if $n = 7$.
    Let $\phi' \colon E(G) \to \{\red, \blue\}$ be the coloring obtained from $\phi$ by giving the color $\phi(v_1v_{\ceil{n/2}+1})$ to all edges in $E(v_1, Y)$ and the remaining color to all edges in $E(v_2, Y)$.
    It is easy to see that $\phi'$ is a locally irregular 2-edge coloring for $G$, and hence~\ref{it:main-2-4} follows.
    Thus, we may consider $1 \leq d_2 \leq d_1 < 3$.
    If $d_1 = d_2 = 2$ or $d_1 = 2$ and $d_2 = 1$, then let $\phi$ be defined as above and let $\phi'' \colon E(G) \to \{\red, \blue\}$ be the coloring obtained from $\phi$ by giving the color $\phi(v_1v_{\ceil{n/2}+1})$ to all edges in $E(v_1, Y)$ and in $E(v_2, Y)$.
    It is easy to see that $\phi''$ is a locally irregular 2-edge coloring for $G$, and hence~\ref{it:main-2-4} follows.
    The last case is when $d_1 = d_2 = 1$.
    Aided by a computer program, we verified that $\chii(G) = 3$ in this case, and thus~\ref{it:main-2-5} follows. 

    At last, suppose that $n \in \{8, 9\}$.
    If $d_4 \geq 1$, then $\chii(G) = 2$ by Lemma~\ref{lemma:splitlarge-2}, and hence~\ref{it:main-2-4} follows.
    So we may assume that $d_4 = 0$.
    If $d_1 < 4$ and $d_2 = 0$, then $\chii(G) = 3$ by Lemma~\ref{lemma:splitlarge-1} and thus~\ref{it:main-2-5} follows.
    Thus, $d_1 \geq 4$ or $d_2 \geq 1$.
    If $d_1 \geq 4$, then let $\phi \colon E(G[X]) \to \{\red,\blue\}$ be a normal coloring to the sequence $(v_2, v_4, v_5, v_1, v_6, v_7, v_8, v_3)$ if $n = 8$, or the sequence $(v_2, v_3, v_4, v_5, v_1, v_6, v_7, v_8, v_9)$ if $n = 9$.
    Let $\phi' \colon E(G) \to \{\red, \blue\}$ be the coloring obtained from $\phi$ by giving the color $\phi(v_1v_6)$ to all edges in $E(v_1, Y)$ and the remaining color to all edges in $E(v_2, Y)$ and in $E(v_3, Y)$.
    It is easy to see that $\phi'$ is a locally irregular 2-edge coloring for $G$, and hence~\ref{it:main-2-4} follows.
    Thus, we may consider $0 \leq d_3 \leq d_2 \leq d_1 < 4$ and also $d_2 \geq 1$.
    Note that there are 16 combinations of values for $d_1$, $d_2$, and $d_3$, which we divide in four cases.
    For each of the first three cases, we obtain a normal coloring $\phi \colon E(G[X]) \to \{\red, \blue\}$ to some sequence $S$ of vertices, which we describe below:
    \begin{enumerate}
        \item if $d_1 = 3$, then we have
        \begin{enumerate}
	        	\item $S=(v_2, v_4, v_5, v_1, v_6=w, v_7, v_8, v_3)$ if $n = 8$, or
		\item $S=(v_3, v_4, v_5, v_6, v_1, v_7=w, v_8, v_9, v_2)$ if $n = 9$.
        \end{enumerate}

        \item if $d_1 = d_2 = 2$, then we have
        \begin{enumerate}
        		\item $S=(v_4, v_2, v_5, v_1, v_6=w, v_7, v_8, v_3)$ if $n=8$, or 
		\item $S=(v_3, v_4, v_5, v_6, v_1, v_7=w, v_8, v_2, v_9)$ if $n=9$.
        \end{enumerate}

        \item if $d_1 = 2$ and $d_2 = d_3 = 1$, then we have 
                \begin{enumerate}
        		\item $S=(v_2, v_3, v_4, v_1, v_5=w, v_6, v_7, v_8)$ if $n=8$, or 
		\item $S=(v_4, v_5, v_6, v_7, v_1, v_8=w, v_9, v_2, v_3)$ if $n=9$.
        \end{enumerate}
    \end{enumerate}
    Then we obtain $\phi'$ from $\phi$ by giving the color $\phi(v_1,v_{\ceil{n/2}+2})$ to all edges in $E(v_1,Y)$ and in $E(v_2,Y)$ and the remaining color to all edges in $E(v_3,Y)$.
    It is not hard to see that $\phi'$ is a locally irregular 2-edge coloring for $G$, and hence~\ref{it:main-2-4} follows for the first three cases.
    The last case considers (i) $d_1 = 2$, $d_2 = 1$, and $d_3 = 0$, (ii) $d_1 = d_2 = d_3 = 1$, and (iii) $d_1 = d_2 = 1$ and $d_3 = 0$.
    Figures~\ref{fig:small-especial-case-1} and~\ref{fig:small-especial-case-2}  show that $\chii(G) = 2$ for (i) and (ii), and hence~\ref{it:main-2-6} holds for these case.
    Figures~\ref{fig:small-especial-case-3}  shows that $\chii(G) = 2$ for (iii) when $n = 9$ and hence~\ref{it:main-2-7} holds for this case. 
    Also aided by a computer program, we verified that $\chii(G) = 3$ for (iii) when $n = 8$, and thus~\ref{it:main-2-5} follows. 
\end{proof}%

\begin{figure}[ht!]
      \centering
      \begin{subfigure}[b]{\linewidth}
        \centering\scalebox{.8}{%

\begin{tikzpicture}[scale=0.8]
    \pgfmathsetmacro{\qtd}{int(10)}
    \pgfmathsetmacro{\raio}{2.7}

    \tikzset{black vertex/.style={circle,draw,minimum size=2mm,inner sep=0pt,outer sep=1pt,fill=black, color=black}}
    \tikzset{square vertex/.style={rectangle,draw,minimum size=2mm,inner sep=0pt,outer sep=1pt,fill=black, color=black}}
    \tikzstyle{red edge}=[line width=1, red]
    \tikzstyle{blue edge}=[line width=1, blue]

    \pgfmathsetmacro{\espaco}{360/10}
    \pgfmathsetmacro{\espacoMetade}{360/(10*2)}
    \pgfmathsetmacro{\rot}{-\espaco - 90 + \espacoMetade}
    
    \foreach \x in {1,...,10} {
        \pgfmathsetmacro{\ang}{\x*\espaco+\rot}
        \ifthenelse{ \equal{\x}{9} } {
          \node (\x) [square vertex] at (\ang:\raio) {}; 
        }{
          \node (\x) [black vertex] at (\ang:\raio) {};
        }
    }
    \draw[red edge] (1) -- (3);
    \draw[red edge] (1) -- (5);
    \draw[red edge] (1) -- (6);
    \draw[red edge] (2) -- (5);
    \draw[red edge] (2) -- (6);
    \draw[red edge] (2) -- (7);
    \draw[red edge] (2) -- (8);
    \draw[red edge] (3) -- (5);
    \draw[red edge] (3) -- (6);
    \draw[red edge] (3) -- (8);
    \draw[red edge] (4) -- (5);
    \draw[red edge] (4) -- (8);
    \draw[red edge] (5) -- (6);
    \draw[red edge] (5) -- (7);
    \draw[red edge] (5) -- (8);
    \draw[red edge] (6) -- (7);
    \draw[red edge] (6) -- (8);
    \draw[blue edge] (1) -- (2);
    \draw[blue edge] (1) -- (4);
    \draw[blue edge] (1) -- (7);
    \draw[blue edge] (1) -- (8);
    \draw[blue edge] (1) -- (9);
    \draw[blue edge] (1) -- (10);
    \draw[blue edge] (2) -- (3);
    \draw[blue edge] (2) -- (4);
    \draw[blue edge] (2) -- (9);
    \draw[blue edge] (3) -- (4);
    \draw[blue edge] (3) -- (7);
    \draw[blue edge] (4) -- (6);
    \draw[blue edge] (4) -- (7);
    \draw[blue edge] (7) -- (8);

    \node (1_label) [ ] at (1 * \espaco+\rot:\raio + 1) {$v_{1}$ ({\color{red} $3$}, {\color{blue} $6$})};
    \node (2_label) [ ] at (2 * \espaco+\rot:\raio + 1) {$v_{2}$ ({\color{red} $4$}, {\color{blue} $4$})};
    \node (3_label) [ ] at (3 * \espaco+\rot:\raio + 1) {$v_{3}$ ({\color{red} $4$}, {\color{blue} $3$})};
    \node (4_label) [ ] at (4 * \espaco+\rot:\raio + 1) {$v_{4}$ ({\color{red} $2$}, {\color{blue} $5$})};
    \node (5_label) [ ] at (5 * \espaco+\rot:\raio + 1) {$v_{5}$ ({\color{red} $7$}, {\color{blue} $0$})};
    \node (6_label) [ ] at (6 * \espaco+\rot:\raio + 1) {$v_{6}$ ({\color{red} $6$}, {\color{blue} $1$})};
    \node (7_label) [ ] at (7 * \espaco+\rot:\raio + 1) {$v_{7}$ ({\color{red} $3$}, {\color{blue} $4$})};
    \node (8_label) [ ] at (8 * \espaco+\rot:\raio + 1) {$v_{8}$ ({\color{red} $5$}, {\color{blue} $2$})};
    \node (9_label) [ ] at (9 * \espaco+\rot:\raio + 1) {$y_{2}$ ({\color{red} $0$}, {\color{blue} $2$})};
    \node (10_label) [ ] at (10 * \espaco+\rot:\raio + 1) {$y_{1}$ ({\color{red} $0$}, {\color{blue} $1$})};
    
\end{tikzpicture}

}
        \hfill  
        \centering\scalebox{.8}{%

\begin{tikzpicture}[scale=0.8]
    \pgfmathsetmacro{\qtd}{int(11)}
    \pgfmathsetmacro{\raio}{2.7}

    \tikzset{black vertex/.style={circle,draw,minimum size=2mm,inner sep=0pt,outer sep=1pt,fill=black, color=black}}
    \tikzset{square vertex/.style={rectangle,draw,minimum size=2mm,inner sep=0pt,outer sep=1pt,fill=black, color=black}}
    \tikzstyle{red edge}=[line width=1, red]
    \tikzstyle{blue edge}=[line width=1, blue]

    \pgfmathsetmacro{\espaco}{360/11}
    \pgfmathsetmacro{\espacoMetade}{360/(11*2)}
    \pgfmathsetmacro{\rot}{-\espaco - 90 + \espacoMetade}
    
    \foreach \x in {1,...,11} {
      \pgfmathsetmacro{\ang}{\x*\espaco+\rot}
      
      \ifthenelse{ \equal{\x}{10} } {
        \node (\x) [square vertex] at (\ang:\raio) {}; 
      }{
        \node (\x) [black vertex] at (\ang:\raio) {};
      }
    }
    
    \draw[red edge] (1) -- (3);
    \draw[red edge] (1) -- (5);
    \draw[red edge] (1) -- (8);
    \draw[red edge] (2) -- (3);
    \draw[red edge] (2) -- (4);
    \draw[red edge] (2) -- (8);
    \draw[red edge] (2) -- (9);
    \draw[red edge] (3) -- (4);
    \draw[red edge] (3) -- (5);
    \draw[red edge] (3) -- (6);
    \draw[red edge] (3) -- (8);
    \draw[red edge] (3) -- (9);
    \draw[red edge] (4) -- (8);
    \draw[red edge] (5) -- (6);
    \draw[red edge] (5) -- (7);
    \draw[red edge] (5) -- (8);
    \draw[red edge] (5) -- (9);
    \draw[red edge] (6) -- (8);
    \draw[red edge] (6) -- (9);
    \draw[red edge] (7) -- (8);
    \draw[red edge] (8) -- (9);
    \draw[blue edge] (1) -- (2);
    \draw[blue edge] (1) -- (4);
    \draw[blue edge] (1) -- (6);
    \draw[blue edge] (1) -- (7);
    \draw[blue edge] (1) -- (9);
    \draw[blue edge] (1) -- (10);
    \draw[blue edge] (1) -- (11);
    \draw[blue edge] (2) -- (5);
    \draw[blue edge] (2) -- (6);
    \draw[blue edge] (2) -- (7);
    \draw[blue edge] (2) -- (10);
    \draw[blue edge] (3) -- (7);
    \draw[blue edge] (4) -- (5);
    \draw[blue edge] (4) -- (6);
    \draw[blue edge] (4) -- (7);
    \draw[blue edge] (4) -- (9);
    \draw[blue edge] (6) -- (7);
    \draw[blue edge] (7) -- (9);
    
    \node (1_label) [ ] at (1 * \espaco+\rot:\raio + 1) {$v_{1}$ ({\color{red} $3$}, {\color{blue} $7$})};
    \node (2_label) [ ] at (2 * \espaco+\rot:\raio + 1) {$v_{2}$ ({\color{red} $4$}, {\color{blue} $5$})};
    \node (3_label) [ ] at (3 * \espaco+\rot:\raio + 1) {$v_{3}$ ({\color{red} $7$}, {\color{blue} $1$})};
    \node (4_label) [ ] at (4 * \espaco+\rot:\raio + 1) {$v_{4}$ ({\color{red} $3$}, {\color{blue} $5$})};
    \node (5_label) [ ] at (5 * \espaco+\rot:\raio + 1) {$v_{5}$ ({\color{red} $6$}, {\color{blue} $2$})};
    \node (6_label) [ ] at (6 * \espaco+\rot:\raio + 1) {$v_{6}$ ({\color{red} $4$}, {\color{blue} $4$})};
    \node (7_label) [ ] at (7 * \espaco+\rot:\raio + 1) {$v_{7}$ ({\color{red} $2$}, {\color{blue} $6$})};
    \node (8_label) [ ] at (8 * \espaco+\rot:\raio + 1) {$v_{8}$ ({\color{red} $8$}, {\color{blue} $0$})};
    \node (9_label) [ ] at (9 * \espaco+\rot:\raio + 1) {$v_{9}$ ({\color{red} $5$}, {\color{blue} $3$})};
    \node (10_label) [ ] at (10 * \espaco+\rot:\raio + 1) {$y_{2}$ ({\color{red} $0$}, {\color{blue} $2$})};
    \node (11_label) [ ] at (11 * \espaco+\rot:\raio + 1) {$y_{1}$ ({\color{red} $0$}, {\color{blue} $1$})};
    
\end{tikzpicture}
}
        \caption{}\label{fig:small-especial-case-1}
      \end{subfigure}%
      
      \begin{subfigure}[b]{\linewidth}
        \centering\scalebox{.8}{%

\begin{tikzpicture}[scale=0.8]
    \pgfmathsetmacro{\qtd}{int(9)}
    \pgfmathsetmacro{\raio}{2.7}

    \tikzset{black vertex/.style={circle,draw,minimum size=2mm,inner sep=0pt,outer sep=1pt,fill=black, color=black}}
    \tikzset{square vertex/.style={rectangle,draw,minimum size=2mm,inner sep=0pt,outer sep=1pt,fill=black, color=black}}
    \tikzstyle{red edge}=[line width=1, red]
    \tikzstyle{blue edge}=[line width=1, blue]

    \pgfmathsetmacro{\espaco}{360/9}
    \pgfmathsetmacro{\espacoMetade}{360/(9*2)}
    \pgfmathsetmacro{\rot}{-\espaco - 90 + \espacoMetade}
\foreach \x in {1,...,9} {
    \pgfmathsetmacro{\ang}{\x*\espaco+\rot}
    
    \ifthenelse{ \equal{\x}{9} } {
      \node (\x) [square vertex] at (\ang:\raio) {}; 
    }{
      \node (\x) [black vertex] at (\ang:\raio) {};
    }
    
    \node (\x) [black vertex] at (\ang:\raio) {};
}
\draw[red edge] (1) -- (4);
\draw[red edge] (1) -- (6);
\draw[red edge] (1) -- (8);
\draw[red edge] (2) -- (3);
\draw[red edge] (2) -- (4);
\draw[red edge] (2) -- (6);
\draw[red edge] (2) -- (8);
\draw[red edge] (3) -- (6);
\draw[red edge] (3) -- (7);
\draw[red edge] (3) -- (8);
\draw[red edge] (3) -- (9);
\draw[red edge] (4) -- (6);
\draw[red edge] (4) -- (7);
\draw[red edge] (4) -- (8);
\draw[red edge] (5) -- (8);
\draw[red edge] (6) -- (7);
\draw[red edge] (6) -- (8);
\draw[red edge] (7) -- (8);
\draw[blue edge] (1) -- (2);
\draw[blue edge] (1) -- (3);
\draw[blue edge] (1) -- (5);
\draw[blue edge] (1) -- (7);
\draw[blue edge] (1) -- (9);
\draw[blue edge] (2) -- (5);
\draw[blue edge] (2) -- (7);
\draw[blue edge] (2) -- (9);
\draw[blue edge] (3) -- (4);
\draw[blue edge] (3) -- (5);
\draw[blue edge] (4) -- (5);
\draw[blue edge] (5) -- (6);
\draw[blue edge] (5) -- (7);
        \node (1_label) [ ] at (1 * \espaco+\rot:\raio + 1) {$v_{1}$ ({\color{red} $3$}, {\color{blue} $5$})};
        \node (2_label) [ ] at (2 * \espaco+\rot:\raio + 1) {$v_{2}$ ({\color{red} $4$}, {\color{blue} $4$})};
        \node (3_label) [ ] at (3 * \espaco+\rot:\raio + 1) {$v_{3}$ ({\color{red} $5$}, {\color{blue} $3$})};
        \node (4_label) [ ] at (4 * \espaco+\rot:\raio + 1) {$v_{4}$ ({\color{red} $5$}, {\color{blue} $2$})};
        \node (5_label) [ ] at (5 * \espaco+\rot:\raio + 1) {$v_{5}$ ({\color{red} $1$}, {\color{blue} $6$})};
        \node (6_label) [ ] at (6 * \espaco+\rot:\raio + 1) {$v_{6}$ ({\color{red} $6$}, {\color{blue} $1$})};
        \node (7_label) [ ] at (7 * \espaco+\rot:\raio + 1) {$v_{7}$ ({\color{red} $4$}, {\color{blue} $3$})};
        \node (8_label) [ ] at (8 * \espaco+\rot:\raio + 1) {$v_{8}$ ({\color{red} $7$}, {\color{blue} $0$})};
        \node (9_label) [ ] at (9 * \espaco+\rot:\raio + 1) {$y_{1}$ ({\color{red} $1$}, {\color{blue} $2$})};
\end{tikzpicture}
}
        \hfill
        \centering\scalebox{.8}{%

\begin{tikzpicture}[scale=0.8]
    \pgfmathsetmacro{\qtd}{int(10)}
    \pgfmathsetmacro{\raio}{2.7}

    \tikzset{black vertex/.style={circle,draw,minimum size=2mm,inner sep=0pt,outer sep=1pt,fill=black, color=black}}
    \tikzset{square vertex/.style={rectangle,draw,minimum size=2mm,inner sep=0pt,outer sep=1pt,fill=black, color=black}}
    \tikzstyle{red edge}=[line width=1, red]
    \tikzstyle{blue edge}=[line width=1, blue]

    \pgfmathsetmacro{\espaco}{360/10}
    \pgfmathsetmacro{\espacoMetade}{360/(10*2)}
    \pgfmathsetmacro{\rot}{-\espaco - 90 + \espacoMetade}
    \foreach \x in {1,...,10} {
        \pgfmathsetmacro{\ang}{\x*\espaco+\rot}
        
        \ifthenelse{ \equal{\x}{10} } {
          \node (\x) [square vertex] at (\ang:\raio) {}; 
        }{
          \node (\x) [black vertex] at (\ang:\raio) {};
        }
    }
    
    \draw[red edge] (1) -- (2);
    \draw[red edge] (1) -- (3);
    \draw[red edge] (1) -- (4);
    \draw[red edge] (1) -- (6);
    \draw[red edge] (1) -- (7);
    \draw[red edge] (1) -- (8);
    \draw[red edge] (1) -- (9);
    \draw[red edge] (1) -- (10);
    \draw[red edge] (2) -- (3);
    \draw[red edge] (2) -- (4);
    \draw[red edge] (2) -- (6);
    \draw[red edge] (2) -- (8);
    \draw[red edge] (2) -- (9);
    \draw[red edge] (2) -- (10);
    \draw[red edge] (3) -- (4);
    \draw[red edge] (3) -- (8);
    \draw[red edge] (3) -- (10);
    \draw[red edge] (4) -- (8);
    \draw[red edge] (5) -- (9);
    \draw[red edge] (6) -- (9);
    \draw[red edge] (7) -- (8);
    \draw[red edge] (8) -- (9);
    \draw[blue edge] (1) -- (5);
    \draw[blue edge] (2) -- (5);
    \draw[blue edge] (2) -- (7);
    \draw[blue edge] (3) -- (5);
    \draw[blue edge] (3) -- (6);
    \draw[blue edge] (3) -- (7);
    \draw[blue edge] (3) -- (9);
    \draw[blue edge] (4) -- (5);
    \draw[blue edge] (4) -- (6);
    \draw[blue edge] (4) -- (7);
    \draw[blue edge] (4) -- (9);
    \draw[blue edge] (5) -- (6);
    \draw[blue edge] (5) -- (7);
    \draw[blue edge] (5) -- (8);
    \draw[blue edge] (6) -- (7);
    \draw[blue edge] (6) -- (8);
    \draw[blue edge] (7) -- (9);

    \node (1_label) [ ] at (1 * \espaco+\rot:\raio + 1) {$v_{1}$ ({\color{red} $8$}, {\color{blue} $1$})};
    \node (2_label) [ ] at (2 * \espaco+\rot:\raio + 1) {$v_{2}$ ({\color{red} $7$}, {\color{blue} $2$})};
    \node (3_label) [ ] at (3 * \espaco+\rot:\raio + 1) {$v_{3}$ ({\color{red} $5$}, {\color{blue} $4$})};
    \node (4_label) [ ] at (4 * \espaco+\rot:\raio + 1) {$v_{4}$ ({\color{red} $4$}, {\color{blue} $4$})};
    \node (5_label) [ ] at (5 * \espaco+\rot:\raio + 1) {$v_{5}$ ({\color{red} $1$}, {\color{blue} $7$})};
    \node (6_label) [ ] at (6 * \espaco+\rot:\raio + 1) {$v_{6}$ ({\color{red} $3$}, {\color{blue} $5$})};
    \node (7_label) [ ] at (7 * \espaco+\rot:\raio + 1) {$v_{7}$ ({\color{red} $2$}, {\color{blue} $6$})};
    \node (8_label) [ ] at (8 * \espaco+\rot:\raio + 1) {$v_{8}$ ({\color{red} $6$}, {\color{blue} $2$})};
    \node (9_label) [ ] at (9 * \espaco+\rot:\raio + 1) {$v_{9}$ ({\color{red} $5$}, {\color{blue} $3$})};
    \node (10_label) [ ] at (10 * \espaco+\rot:\raio + 1) {$y_{1}$ ({\color{red} $3$}, {\color{blue} $0$})};
\end{tikzpicture}
}
        \caption{}\label{fig:small-especial-case-2}
      \end{subfigure}%

      \begin{subfigure}[b]{.5\linewidth}
        \centering\scalebox{.8}{%

\begin{tikzpicture}[scale=0.8]

    \pgfmathsetmacro{\qtd}{int(10)}
    \pgfmathsetmacro{\raio}{2.7}

    \tikzset{black vertex/.style={circle,draw,minimum size=2mm,inner sep=0pt,outer sep=1pt,fill=black, color=black}}
    \tikzset{square vertex/.style={rectangle,draw,minimum size=2mm,inner sep=0pt,outer sep=1pt,fill=black, color=black}}
    \tikzstyle{red edge}=[line width=1, red]
    \tikzstyle{blue edge}=[line width=1, blue]

    \pgfmathsetmacro{\espaco}{360/10}
    \pgfmathsetmacro{\espacoMetade}{360/(10*2)}
    \pgfmathsetmacro{\rot}{-\espaco - 90 + \espacoMetade}

    \foreach \x in {1,...,10} {
      \pgfmathsetmacro{\ang}{\x*\espaco+\rot}
      \node (\x) [black vertex] at (\ang:\raio) {};
      
      \ifthenelse{ \equal{\x}{10} } {
        \node (\x) [square vertex] at (\ang:\raio) {}; 
      }{
        \node (\x) [black vertex] at (\ang:\raio) {};
      }
    }
    
    \draw[red edge] (1) -- (4);
    \draw[red edge] (1) -- (6);
    \draw[red edge] (1) -- (8);
    \draw[red edge] (1) -- (9);
    \draw[red edge] (1) -- (10);
    \draw[red edge] (2) -- (3);
    \draw[red edge] (2) -- (4);
    \draw[red edge] (2) -- (7);
    \draw[red edge] (2) -- (9);
    \draw[red edge] (3) -- (4);
    \draw[red edge] (3) -- (7);
    \draw[red edge] (3) -- (8);
    \draw[red edge] (3) -- (9);
    \draw[red edge] (4) -- (5);
    \draw[red edge] (4) -- (7);
    \draw[red edge] (4) -- (8);
    \draw[red edge] (4) -- (9);
    \draw[red edge] (6) -- (9);
    \draw[red edge] (8) -- (9);
    \draw[blue edge] (1) -- (2);
    \draw[blue edge] (1) -- (3);
    \draw[blue edge] (1) -- (5);
    \draw[blue edge] (1) -- (7);
    \draw[blue edge] (2) -- (5);
    \draw[blue edge] (2) -- (6);
    \draw[blue edge] (2) -- (8);
    \draw[blue edge] (2) -- (10);
    \draw[blue edge] (3) -- (5);
    \draw[blue edge] (3) -- (6);
    \draw[blue edge] (4) -- (6);
    \draw[blue edge] (5) -- (6);
    \draw[blue edge] (5) -- (7);
    \draw[blue edge] (5) -- (8);
    \draw[blue edge] (5) -- (9);
    \draw[blue edge] (6) -- (7);
    \draw[blue edge] (6) -- (8);
    \draw[blue edge] (7) -- (8);
    \draw[blue edge] (7) -- (9);

    \node (1_label) [ ] at (1 * \espaco+\rot:\raio + 1) {$v_{1}$ ({\color{red} $5$}, {\color{blue} $4$})};
    \node (2_label) [ ] at (2 * \espaco+\rot:\raio + 1) {$v_{2}$ ({\color{red} $4$}, {\color{blue} $5$})};
    \node (3_label) [ ] at (3 * \espaco+\rot:\raio + 1) {$v_{3}$ ({\color{red} $5$}, {\color{blue} $3$})};
    \node (4_label) [ ] at (4 * \espaco+\rot:\raio + 1) {$v_{4}$ ({\color{red} $7$}, {\color{blue} $1$})};
    \node (5_label) [ ] at (5 * \espaco+\rot:\raio + 1) {$v_{5}$ ({\color{red} $1$}, {\color{blue} $7$})};
    \node (6_label) [ ] at (6 * \espaco+\rot:\raio + 1) {$v_{6}$ ({\color{red} $2$}, {\color{blue} $6$})};
    \node (7_label) [ ] at (7 * \espaco+\rot:\raio + 1) {$v_{7}$ ({\color{red} $3$}, {\color{blue} $5$})};
    \node (8_label) [ ] at (8 * \espaco+\rot:\raio + 1) {$v_{8}$ ({\color{red} $4$}, {\color{blue} $4$})};
    \node (9_label) [ ] at (9 * \espaco+\rot:\raio + 1) {$v_{9}$ ({\color{red} $6$}, {\color{blue} $2$})};
    \node (10_label) [ ] at (10 * \espaco+\rot:\raio + 1) {$y_{1}$ ({\color{red} $1$}, {\color{blue} $1$})};
        
\end{tikzpicture}
}
        \caption{}\label{fig:small-especial-case-3}
      \end{subfigure}%

      \caption{Each figure illustrates a locally irregular 2-edge coloring. 
      We use a square to denote one or more vertices, for example, in (B) $y_1$ could be a single vertex, a pair of vertices, or a triple of vertices with the same neighborhood of $y_1$. In any possible configuration, the coloring presented in (B) is locally irregular.}\label{fig:esp-small-cases}
\end{figure}
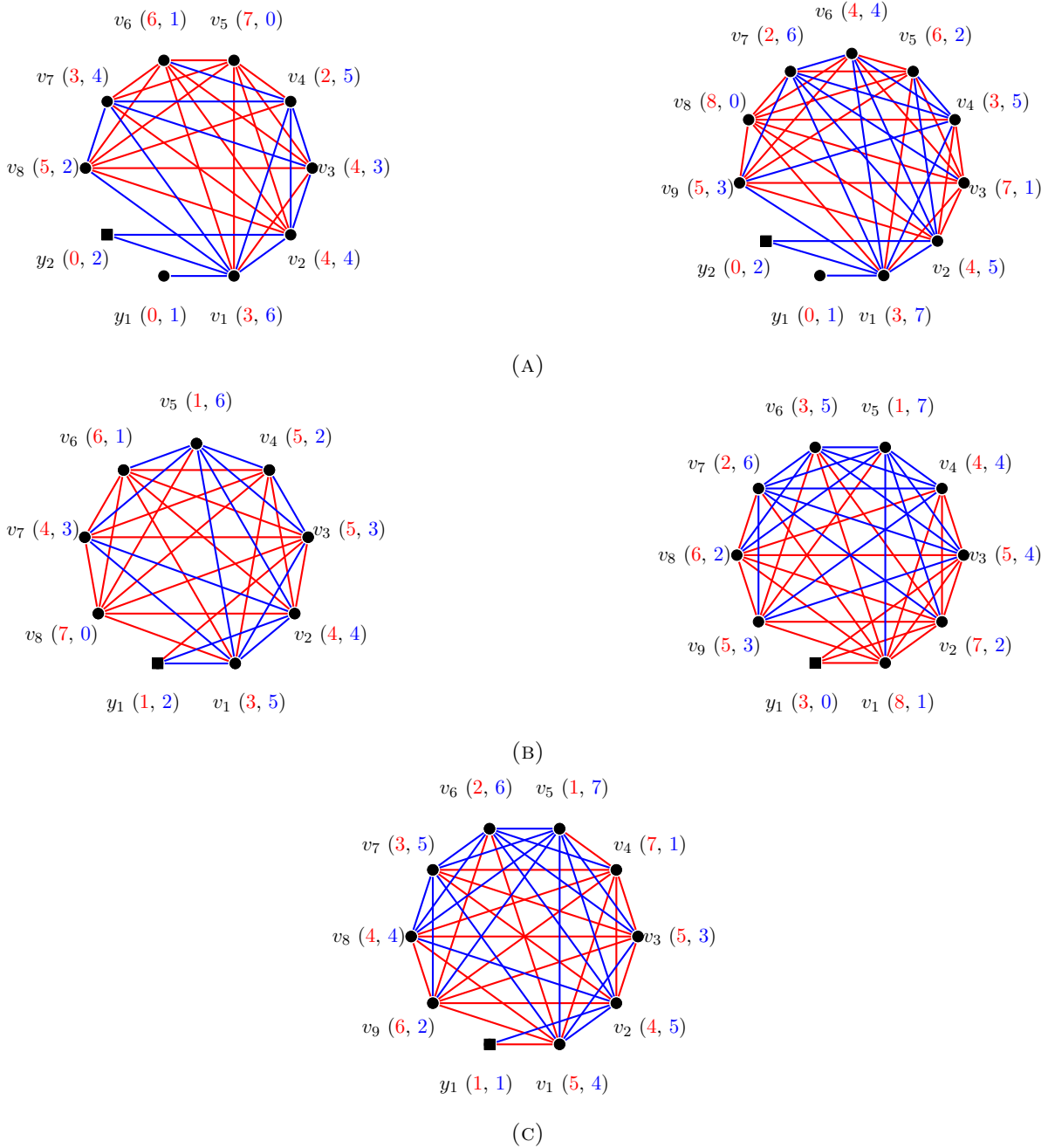

\clearpage

\bibliographystyle{amsplain}
\bibliography{bibliography}

\end{document}